\documentclass{amsart}

\usepackage{amsthm}
\usepackage{amssymb}
\usepackage{amsmath}
\usepackage{amsmath,amscd}
\usepackage{epsfig}
\usepackage{leftidx}

\newcommand{\func}[1]{\ensuremath{\mathrm{#1} \:} }

\newcommand{\arccosh}[0]{\func{arccosh}}

\numberwithin{equation}{section}
\newtheorem{theorem}{Theorem}[section]

\newtheorem{lemma}[theorem]{Lemma}
\newtheorem{proposition}[theorem]{Proposition}
\theoremstyle{definition}
\newtheorem{definition}[theorem]{Definition}
\theoremstyle{remark}
\newtheorem{remark}[theorem]{Remark}

\newcommand{\e}[0]{\mathbf e}
\newcommand{\RRR}[0]{\mathsf T}
\newcommand{\Real}[0]{\mathbb R}

\title[Logarithmically spiraling helicoids]{Logarithmically spiraling  helicoids}

\author{Christine Breiner}
\address{Department of Mathematics, Fordham University, Bronx, NY 10458}
\email{cbreiner@fordham.edu}

\author{Stephen J. Kleene}
\address{Department of Mathematics, MIT, Cambridge, MA 02139}
\email{skleene@math.mit.edu}
\thanks{C. Breiner was supported in part by NSF grant DMS-1308420 and an AMS-Simons Travel Grant. S.J. Kleene was partially supported by NSF grant DMS-1004646.}
\begin{document}
\maketitle

\begin{abstract}
We construct helicoid-like embedded minimal disks with axes along self-similar curves modeled on logarithmic spirals. The surfaces have a self-similarity inherited from the curves and the nature of the construction. Moreover, inside of a ``logarithmic cone'', the surfaces are embedded.

MSC 53A05, 53C21.
\end{abstract}

\section{Introduction}

In this article we construct helicoid-like embedded minimal disks with axes modeled on a class of embedded self-similar curves called logarithmic spirals. Logarithmic spirals are solutions $\gamma(z)$ to the initial value problem
\begin{align} \notag
\kappa(z) = \kappa_0 e^{ - \xi z}, \quad \tau (z) = \tau_0 e^{- \xi z}
\end{align}
where $\tau$ and $\kappa$ denote the torsion and curvature of the unknown curve, respectively, and where the constant $\xi$ controls the rate of exponential propagation. 

\begin{figure}[htb]\label{spirals}
\begin{center}
\includegraphics[height=2in,angle= 0 ]{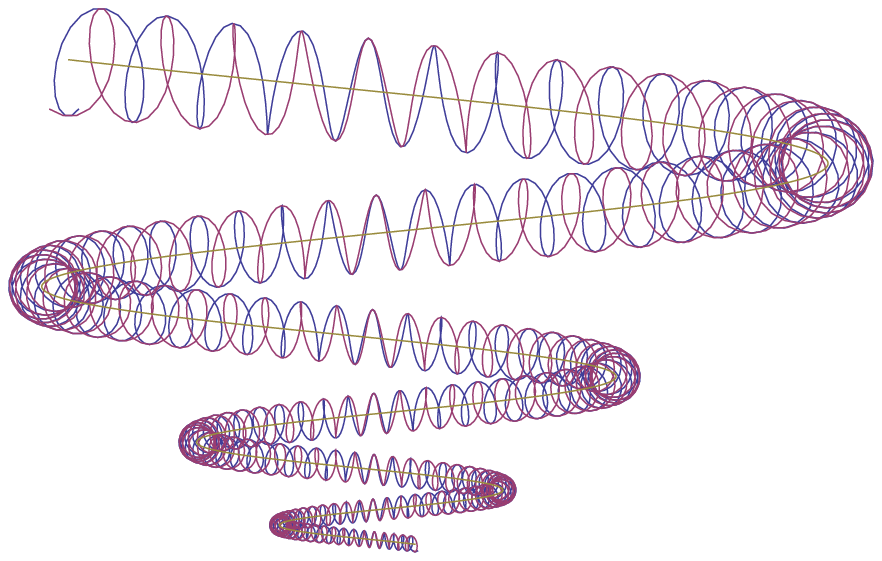} \includegraphics[height=2in,angle= 0 ]{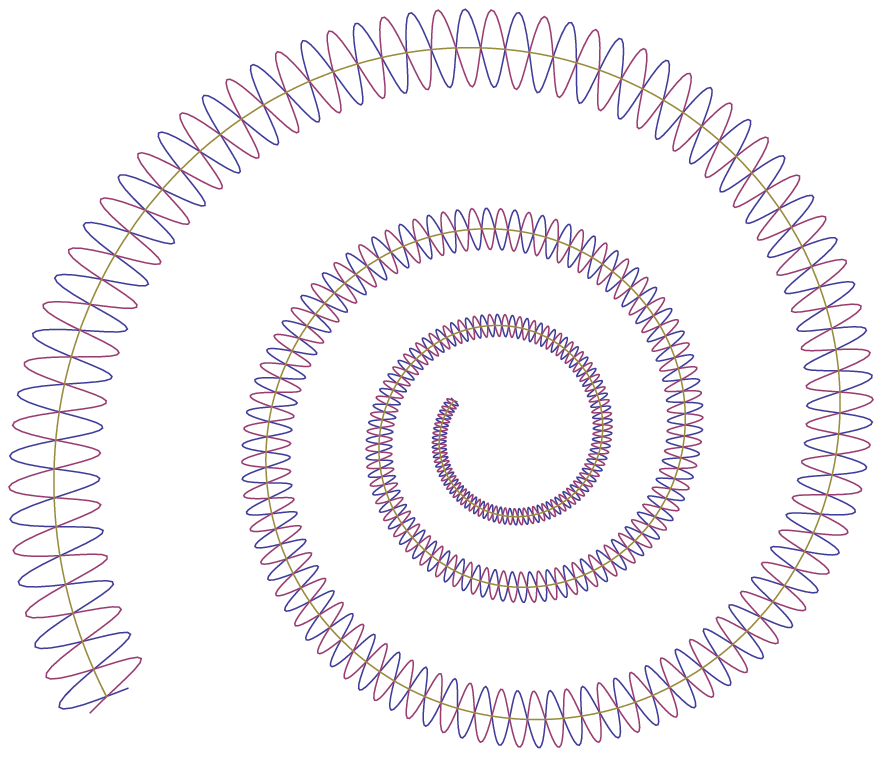}
\caption{Two examples of logarithmic spirals with the boundaries of the corresponding surfaces $S_\delta$.}\end{center}
\end{figure}Our main theorem can then be stated as:

\begin{theorem} \label{TheoremForTheCommonMan}
For each $\delta>0$ sufficiently small, there exist minimal disks $S_\delta$ such that
\begin{enumerate}
\item Up to rigid motion, the surfaces $S_\delta$ exhibit a discrete dilation invariance. 
\item  The surfaces $S_\delta$ are embedded inside an open set $\hat{T}$ containing $\gamma(z)$ independent of $\delta$. As $\delta \rightarrow 0$ the surfaces $S_\delta$ converge smoothly away from $\gamma(z)$ to a foliation  of $\hat{T}$ by planes orthogonal to $\gamma(z)$.
\end{enumerate}
\end{theorem}\noindent This result follows immediately from a much more general theorem, which we state below.

Our construction is related to the following more general question: Given a singly periodic minimal surface $\Sigma$, a smooth curve $\gamma$ and a non-negative function $\lambda: \gamma \rightarrow \Real^+$, can one obtain a minimal surface in a tubular neighborhood of $\gamma$ by bending $\Sigma$ along $\gamma$ and scaling by $\lambda$? If so, what are the restrictions on the scale function $\lambda$?  This and related problems arise in several contexts, including gluing constructions for minimal surfaces and the theory of Colding-Minicozzi type laminations, and in certain special cases is well understood. One of the  simplest non-trivial cases is that of a constant scale function $\lambda \equiv c $ and a periodic curve $\gamma$, which arises naturally in highly symmetric gluing constructions such as \cite{KapMin}, where Kapouleas has developed the theory extensively.  In this case, the  candidate surface can be constructed with the same periodicity as the underlying curve and the problem descends to a compact quotient of the periodic minimal surface $\Sigma$ which simplifies the analysis considerably.     
 For constructions which do not descend to compact quotients, indirect methods are typically employed. For example, Meeks and Weber in  \cite{MW} use analytic methods  to construct helicoidal minimal surfaces in tubes along curves essentially corresponding to the case that $\lambda$ is a constant. In \cite{HW3}   Hoffmann and White employ variational techniques to construct minimal laminations in tubes with singularities on prescribed compact subsets of curves. The second author, in \cite{Kl}, uses the Weierstrass representation to prove the same result as in \cite{HW3}. For additional constructions in the same spirit, see also \cite{CMPVNP,BDeanPaper,SidPaper}. 
 
 One of the main difficulties in approaching such constructions with more direct methods lies in understanding  linear analysis on periodic minimal surfaces that does not descend to any compact quotient. Directly perturbing a bent minimal surface back to minimality essentially entails inverting the stability operator in suitable function spaces and, if reasonable bounds are expected, then it must be possible to impose orthogonality of the error term to a continuous spectrum of small eigenvalues. A second difficulty is appropriately measuring the size of the error term. When the analysis descends to compact quotients--for example in the case that $\lambda$ is constant and $\gamma$ is a circle--H\"older norms are equivalent to Sobolev norms. In the general case, however, there is no reason to expect the initial error will lie in $L^2$ globally so that standard PDE methods have to be carefully employed. 
 As a heuristic, it is  reasonable to approach the general problem by first asking: how much of the geometry of the bent and scaled surface is self-similar?   In the case of a constant scale function, curves of constant torsion and curvature (spirals) are natural candidates for the best approximating self-similar problem. When the derivative of the scale function is non-vanishing--corresponding in our construction to $\xi \neq 0$--the most natural candidates are the logarithmic spirals,  where the rate $\xi$ of propagation is proportional to the derivative of the scale function. Our goal in this paper is then to formulate precisely  the  broadest class of self-similar problems for which the analysis descends to compact quotients and to solve in the case that the periodic minimal surface is the helicoid.  The helicoid is a natural choice, due to the simplicity of its topology and geometry. However, we emphasis that our methods do not rely inherently on the trivial topology of the helicoid, and should be applicable to a broad class of periodic minimal surfaces, including the singly-periodic Scherk surfaces.

\subsection{Precise statement of the main theorem}

Given $\RRR = (\RRR_{i  j})$ a fixed anti-symmetric $3 \times 3$ matrix, $\xi \in \Real$, and $\delta>0$ there exists a ``logarithmic spiral'', $\gamma (z)$, parametrized by $z \in \mathbb{R}$, with the property
\begin{align}\label{frame_derivative}
\gamma' (z) = e^{\delta \xi  z} \e_3 (z), \quad \e' (z) = \delta \RRR \e(z).
\end{align}
Here $\e = \{\e_1, \e_2, \e_3 \}$ is an orthonormal frame along $\gamma$. The curvature and torsion of $\gamma$ are completely determined by $\RRR$ and the frame and $\kappa = \delta \kappa_0e^{-\delta \xi z}, \tau = \delta \tau_0 e^{-\delta \xi z}$ where $\delta\kappa_0, \delta\tau_0$ correspond to the curvature and torsion of $\gamma$ at $z=0$. We use these curves to define a map $M: \Real^3 \to \Real^3$ given by
\begin{align} \label{Mdef}
M(x, y, z) : = \gamma (z) + e^{\delta   \xi z}\left\{ x \e_1 + y \e_2\right\}.
\end{align}
We then  construct helicoid-like minimal surfaces as graphs over a surface obtained by composing $M$ with the conformal parameterization of the helicoid\begin{align} \label{Fdef}
F(s, \theta) = \sinh(s) \sin (\theta) \e_x + \sinh(s) \cos (\theta) \e_y +  \theta \e_z.
\end{align}The geometry of the surfaces possess a periodicity inherited from properties of the curve $\gamma$ and the fact that we solve the problem on periodic function spaces.  When $\delta>0$ is sufficiently small, the map $M$ is a diffeomorphism of a tube along the $z$-axis, with radius comparable to $1/\delta$, onto its image. 
Properties of $\gamma$ and the diffeomorphism imply that as $\kappa_0 \to 0, \tau_0 \to 0, \delta\to 0$, the embedded component of $M\circ F$ about $\gamma$ converges to a rigid motion of the helicoid.  The more general theorem is then:
\begin{theorem} \label{RegularTheorem}Given a non-trivial anti-symmetric matrix $\RRR$ and a  constant $\xi \in \Real$, there exist $\epsilon_1 > 0$ and $\delta_0 > 0$ so that: For any $0<\delta <\delta_0/|\xi|$ and $\ell > 16$  satisfying $\delta (1 + |\RRR|+  |\xi| )\ell  \leq \epsilon_1$, there exists a curve $\gamma$ satisfying \eqref{frame_derivative} and a  periodic function $u (s, \theta): \Lambda \subset \mathbb{R}^2 \rightarrow \mathbb{R}$ where $\Lambda := \{(s, \theta) : \cosh(s) 
\leq \ell/4 \}$ such that the normal graph over 
\begin{align} \label{Gdef}
G(s, \theta)  : = M \circ F (s, \theta)
\end{align}
 by $w (s, \theta) : = e^{\delta \xi  \theta} u (s, \theta)$ is an immersed minimal disk with boundary. 
 
Moreover, the surface is embedded if, for $\rho_0 = \sqrt{\kappa_0^2 +\tau_0^2}$,
\begin{align} \notag
\ell \leq \frac{1}{\delta \xi} \left( \frac{e^{\pi \xi/ \rho_0} - 1}{e^{\pi \xi/\rho_0} + 1}\right) \sqrt{ \frac{\tau_0^2 + \xi^2}{\rho_0^2 + \xi^2}}.
\end{align}
\end{theorem}
For the class of curves satisfying \eqref{frame_derivative}, as long as $\RRR$ is non-trivial and $\xi \neq 0$, the global isometry group of the curve $\gamma$ is generated by the action $e^{\delta \xi \tilde \theta }R_{\delta \rho_0  \tilde \theta}$ where $R_t$ corresponds to a rotation in the $xy$-plane by angle $t$. The isometry of $\gamma$ thus implies a periodicity of the form
\[
\gamma(\theta + 2\pi/(\delta \rho_0)) = e^{2\pi \xi /\rho_0}\gamma(\theta)\, \text{      or      }\, \gamma(\theta + 2\pi) =e^{2\pi \delta \xi} R_{2\pi \delta\rho_0}\gamma(\theta).
\]

While not every surface constructed by Theorem \ref{RegularTheorem} exhibits the discrete dilation property described in Theorem \ref{TheoremForTheCommonMan}, when $\{\e_1, \e_2, \e_3\}$ corresponds to the Frenet frame on $\gamma$, the immersion $G$ in this case satisfies
\[
G(s, \theta + 2\pi) = e^{2\pi \xi\delta} R_{2\pi \delta \rho_0} G(s,\theta).
\]Since we solve the problem for functions $u$ with $2\pi$ periodicity in $\theta$, the normal graphs over surfaces satisfying the previous property will exhibits the isometry:
 \[
 G_w(s, \theta + 2\pi) = e^{2\pi \delta\xi }R_{2\pi \delta \rho_0}G_w(s,\theta).
 \]

\subsection{Structure of the article}
In Section  \ref{Preliminaries}, we introduce various preliminary notions that will be needed throughout the article. In Subsection \ref{NotationForNorms}, we introduce notation for norms and function spaces, and in Subsection \ref{EstimatingHomogeneousQuantities} we formalize the process of obtaining weighted $C^{k, \alpha}$ estimate for a broad class of homogeneous quantities. This formalization is extremely useful and helps streamline the presentation since almost every geometric quantity we estimate--the mean curvature, e.g.--is such a quantity. The discussion generalizes a similar one in \cite{BK}. 

In Section \ref{LogarithmicSpirals}, we define a three parameter family of embedded self-similar curves which can be thought of as solutions to the problem of specifying curves with initial curvature, torsion, and rate of exponential parametrization. In the case of vanishing torsion, the traces  of these curves constitute the family of \emph{logarithmic spirals}, and in accordance with this terminology we refer to the whole family as \emph{logarithmic spirals}.  We then show that the curves  satisfying \eqref{frame_derivative} are logarithmic spirals with curvature and torsion proportional to $\delta e^{- \delta \xi  z}$. From this we conclude that  the map $M$ in (\ref{Mdef}) is a diffeomorphism from a tubular neighborhood of the $z$ axis, with radius proportional to $\delta^{-1}$, onto its image, which we call a \emph{logarithmic cone}. Logarithmic cones, like the spirals on which they are modeled, have a one-dimensional global isometry group, and in the case that $|\RRR| = 0$ (so that $\gamma$ is a straight line) agree up to a similarity transformation with a solid positive cone about the $z$-axis.

In Section \ref{GeometricQuantitiesOnG} we estimate local perturbations of geometric quantities on $G$, where here we heavily exploit the group action on $\gamma$.  The maps $M$ in \eqref{Mdef} are locally, after modding out by dilations and rigid motions,  the identity map plus an exponentially growing perturbation term. In this normalization the perturbation term is periodic, which allows us to consider the problem on one fundamental domain of the helicoid. Roughly speaking, if $u$ is $2\pi$ periodic in $\theta$ and $H[u]$ is the mean curvature  of a normal graph over $G$ by the function $u$, then the function $Q[u] : = e^{\delta \xi  \theta}\cosh^2(s) H[e^{\delta \xi \theta} u]$ is also a $2 \pi$ periodic function in $\theta$. In this way, the analysis descends to cylinders of finite (but not uniformly bounded) length.

We solve the linear problem in Section \ref{TheLinearProblem}, where the main result appears in Proposition \ref{FlatTotalInverse}. This result proves the invertibility of the stability operator $\mathcal{L}_F$ on flat cylinders of length $\arccosh(\ell)$, for appropriately modified inhomogeneous terms. The strategy is motivated by the work of Kapouleas in various gluing problems \cite{KapAnn,KapWente,KapMin}. A novel feature of this work is the decomposition of the inhomogeneous term $E$ into a $\theta$-independent function and a function $\mathring E$, which has ``zero average on meridian circles'' (see Definition \ref{ZeroMeridianAverageSpaces}). The $\theta$-independent function can be inverted via direct integration, so the main work is to invert $\mathring E$. We first prove, in Proposition \ref{flat_inverse}, that the operator has a bounded inverse in exponentially weighted H\"older spaces  $\mathcal X^k$ supported on $\Lambda$, 
as long as $\mathring E$ is ``orthogonal'' to a three dimensional kernel 
 which is spanned by the translational Killing fields. The weighting allows for a rate of exponential growth of power $3/4$ though, in fact, anything growth rate less than $1$ also works. Here orthogonal  means $L^2$-orthogonal with respect to the pull-back of $F$ to the sphere under the Gauss map.  The Gauss map for minimal surfaces is conformal with conformal factor $\frac{|A|^2}{2}$ and in the case of the helicoid descends to the quotient as a conformal diffeomorphism onto  the sphere minus the north and south pole (corresponding to the asymptotic normal along both ends of the helicoid). In this way, the study of the stability operator $\mathcal{L}_F$ on the helicoid in $\theta$-periodic function spaces can be understood as the study of $\Delta_{\mathbb S^2} + 2$ on the sphere.  
 
 To modify $\mathring E$, we define functions  $u_x$ and $u_y$, which we can control geometrically, and whose graphs over $F$ can be used to prescribe the kernel content of the mean curvature. In this way, we orthogonalize the error terms for which we are solving and are able to apply Proposition \ref{flat_inverse} to the general setting. In articles by Kapouleas and his many coauthors (see \cite{KapBreiner,HaskKap,KapYang} among others), the functions $\mathcal L_F u_x, \mathcal L_F u_y$ are referred to as the ``substitute kernel''. Notice that while the space of translational Killing fields is three dimensional, the space of modifications is only two dimensional. The $\theta$-independence of the third translation function and the averaging property of $\mathring E$ immediately guarantee the projection of $\mathring E$ in this direction is always zero. The functions $u_x, u_y$ grow exponentially at a rate proportional to $\cosh{(s)}$, which is faster than the allowable growth rate in the space $\mathcal{X}^k$ ($\cosh^{3/4}(s)$), so that a bounded inverse does not extend to the full H\"older space $\mathcal{X}^k$. 

In Section \ref{FindingExactSolutions}, we define a map from an appropriate Banach space and show the estimates are sufficient to invoke Schauder's fixed point theorem. Every point in the Banach space corresponds to a triple $(v,b_x,b_y)$ where $v \in \mathcal X^2$ and $b_x, b_y \in \Real$. For a fixed point, setting $f =  e^{\delta \xi  \theta}\left(v + b_x u_x + b_y u_y \right)$, in Section \ref{MainTheorem} we demonstrate that $G_f$ is an embedded minimal disk. We define the Banach space so that
\begin{align}\notag
|b_x|+ | b_y| \leq \zeta \delta |\RRR| \ell^{1/4}, \quad v \approx \zeta \delta |\RRR|  \ell^{1/4}\cosh^{3/4} (s)
\end{align}
where $\zeta$ is a fixed constant which we must choose sufficiently large. Since $|s| \leq \arccosh (\ell/4)$ on $\Lambda$, $v$ is then bounded by  a constant uniformly proportional to $\zeta \delta |\RRR| \ell$, which we can keep arbitrarily small by taking $\delta$ small.  As $u_x, u_y$ grow exponentially, the estimates for these functions are of a weaker form:
\begin{align}\notag
\sup_\Lambda |b_x u_x| \approx \zeta \delta |\RRR| \ell^{1/4}\cosh(s) \leq \zeta \delta |\RRR| \ell^{5/4}.
\end{align}This bound and the previous give some indication of the constraints that fix an upper bound $\delta$.

With these estimates in hand, we use properties of the map $M$ and the helicoid embedding $F$ to prove embeddedness in Section \ref{MainTheorem}. We first prove that graphs over $F$ by $\left(v + b_x u_x + b_y u_y\right) \nu_F + \mathbf x$ are small on $\Lambda$, if the norms on $v, b_x, b_y$ are sufficiently small and $\mathbf x$ has small $C^1$ norm. The smallness depends only on properties of the helicoid. We then use properties of $M$ to prove that $G_f$ can be locally described as such a graph over $F$. The diffeomorphism property for $M$ then implies embeddedness.
 
\section{Preliminaries} \label{Preliminaries}
\subsection{Notation and conventions}\label{NotationForNorms}
Throughout this paper we make extensive use of cut-off functions, and we adopt the following notation:  Let $\psi_0:\Real \to [0,1]$ be a smooth function such that
\begin{enumerate}
 \item $\psi_0$ is non-decreasing
\item $\psi_0 \equiv 1$ on $[1,\infty)$ and $\psi_0 \equiv 0$ on $(-\infty, -1]$
\item $\psi_0-1/2$ is an odd function.
\end{enumerate}
For $a,b \in \Real$ with $a \neq b$, let $\psi[a,b]:\Real \to [0,1]$ be defined by $\psi[a,b]=\psi_0 \circ L_{a,b}$ where $L_{a,b}:\Real \to \Real$ is a linear function with $L(a)=-3, L(b)=3$.
Then $\psi[a,b]$ has the following properties:
\begin{enumerate}
 \item $\psi[a,b]$ is weakly monotone.
\item $\psi[a,b]=1$ on a neighborhood of $b$ and $\psi[a,b]=0$ on a neighborhood of $a$.
\item $\psi[a,b]+\psi[b,a]=1$ on $\Real$.
\end{enumerate}

\begin{definition}
Given a function $u \in C^{j, \alpha} (D)$, where $D \subset \Real^m$, the   $(j, \alpha)$ \emph{localized H\"older norm} is given by
\begin{align*}
\| u\|_{j, \alpha} (p) : = \| u : C^{j, \alpha} (D \cap B_1 (p))\|.
\end{align*}
We let $C^{j, \alpha}_{loc} (D)$ denote the space of functions for which $ \|- \|_{j, \alpha}$ is  pointwise finite. 
\end{definition}

\begin{definition}
Given  a positive function $f: D \rightarrow \Real$,  we let the space $C^{j, \alpha} (D, f)$ be the space of functions for which the \emph{weighted norm} $\| -: C^{j, \alpha} (D, f)\|$  is finite, where we take
\begin{align*}
\| u : C^{j, \alpha} (D, f) \| : = \sup_{p \in D} f(p)^{-1} \| u\|_{j, \alpha} (p)
\end{align*}
\end{definition}

\begin{definition}
Let $\mathcal{X}$ and $\mathcal{Y}$ be two Banach spaces with norms $\| - : \mathcal{X}\|$ and $\| - : \mathcal{Y}\|$, respectively. Then $\mathcal{X} \cap \mathcal{Y}$ is naturally a Banach space with norm $\| -:  \mathcal{X} \cap \mathcal{Y}\|$ given by
\begin{align}\notag
\| f : \mathcal{X} \cap \mathcal{Y} \| = \| f: \mathcal{X}\| + \| f : \mathcal{Y}\|. 
\end{align}
\end{definition}
Let $\mathcal{X}$ be a Banach space with norm $\| - : \mathcal{X}\|$ and suppose $S \subset \mathcal X$. For convenience, throughout the paper we will sometimes write $\| - : S\|$, where for any $f \in S$ we simply let
\[
\|f:S\|:= \|f: \mathcal X\|.
\]

\subsection{Estimating  homogeneous quantitites} \label{EstimatingHomogeneousQuantities}
Let $E$ be the Euclidean space  $E : = E^{(1)} \times E^{(2)} =   \Real^{3 \times 2} \times \Real^{3 \times 4}$. We denote points of $E$ by $ \underline{\nabla}  = (\nabla, \nabla^2)$, where
\begin{align*}
\nabla = (\nabla_1, \nabla_2) \quad \nabla^2 = (\nabla^2_{1  1}, \nabla^2_{2  2}, \nabla^2_{1  2}, \nabla^2_{2  1}).
\end{align*}
We then consider functions $\Phi(\underline{\nabla})$ on $E$ with the property 
\begin{align*}
\Phi (c \underline{\nabla}) = c^d \Phi (\underline{\nabla})
\end{align*}
for real numbers $c$ and $d$. We call such a function a \emph{homogeneous function of degree $d$}.   It is straightforward to verify that a homogeneous degree $d$ function has the property that its $j^{th}$ derivative $D^{(j)} \Phi $ is homogeneous degree $d - j$.
\begin{remark}
We will assume throughout this section that all functions $\Phi:E \to \Real$ refer to smooth functions which are homogeneous of degree $d$. We also presume such $\Phi$ are uniformly bounded in any $C^k$ on compact subsets of the space
\begin{align*}
E_0 : = \{\underline{\nabla} \in E : \mathfrak{a} (\nabla) \neq 0  \}.
\end{align*}
\end{remark}Notice $E$ is just a Euclidean space so for any $V \in E$, we make the identification $T_VE =E$.
We extend this for each $k \in \mathbb Z^+$ and observe that $D^{(k)}\Phi(V):E^k \to \Real$. For clarity we provide the following definition.
\begin{definition}
Let $k \in \mathbb Z^+$, $V, W_1, \dots W_k \in E$. Then
\begin{equation}\notag
\left.D^{(k)}\Phi\right|_V(W_1 \otimes \dots \otimes W_k):=D^{(k)}\Phi(V)(W_1 \otimes \dots \otimes W_k).
\end{equation}For brevity, we denote the $k$-th tensor product of $W$ with itself by
\[
\otimes^{(k)} W:= W\otimes \dots \otimes W.
\]
\end{definition}
\begin{definition}
 Given an immersion $\phi : D \subset \Real^2 \rightarrow \Real^3$, we set $\underline{\nabla} [\phi] :  = (\nabla \phi, \nabla^2 \phi)$. A \emph{homogeneous quantity of degree $d$} on $\phi$ is then a function of the form $ \Phi[\phi] : = \Phi (\underline{\nabla}[\phi])$ for some homogeneous function $\Phi$ on $E$.
 \end{definition}
   Examples of such functions are the mean curvature, unit normal, components of the metric and its dual, the Christoffel symbols and the coefficients of the Laplace operator for $\phi$ in the domain $D$.

We want to estimate the linear and higher order changes of homogenous quantities along $\phi$ due  to addition of small vector fields. To do this concisely,  we refer to a map $\underline{\nabla} (s, \theta): D \subset \Real^2 \rightarrow E$ as an \emph{immersion} if the quantity
\begin{align} \label{fraka_def}
\mathfrak{a} (\underline{\nabla}) = \mathfrak{a} (\nabla) : = 2 \sqrt{\det{\nabla^T\nabla}}/|\nabla|^2 
\end{align}
is everywhere non-zero, and otherwise we refer to it simply as a \emph{vector field}. 
\begin{lemma}
For a map $\underline \nabla:D \to E$, $|\mathfrak{a}(\nabla)| \leq 1$ with equality if and only if $|\nabla_1| = |\nabla_2|$ and $\nabla_1 \cdot \nabla_2 = 0$. In particular, $\mathfrak{a} (\nabla \phi) = 1$ if and only if $\phi$ is a conformal immersion.
\end{lemma}
\begin{proof}
Since $\mathfrak{a}$ is homogeneous degree $0$ it suffices to consider the case that $|\nabla_1|  = 1$, $|\nabla_2| : = r$ for $r \in [0,1]$. We can then write
\begin{align*}
\mathfrak{a}^2(\nabla)  & = 4\frac{ |\nabla_1|^2 |\nabla_2|^2 -( \nabla_1 \cdot \nabla_2)^2}{(|\nabla_1|^2 + |\nabla_2|^2)^2}  \\ 
& = 4 \frac{r^2 - \cos^2(\theta) r^2}{(1 + r^2)^2} = 4 (1 - \cos^2(\theta)) \frac{r^2}{(1 + r^2)^2}
\end{align*}
For each $\theta$, the right hand side achieves a unique maximum at $r = 1$ with value $(1 - \cos^2 (\theta))$, which gives the claim.
\end{proof}
  
\begin{definition}\label{TRdef}
 Given an immersion $\underline{\nabla}$ and a vector field $\mathcal{E}$, we  set 
\begin{align} \label{TR1}
R_{\Phi, \mathcal{E}}^{(k)} (\underline{\nabla}) : =  \int_0^{1} \frac{(1 - \sigma)^k}{k!} \left. D \Phi^{(k + 1)} \right|_{\underline{\nabla}(\sigma)} (\otimes^{(k + 1)} \mathcal{E}) d \sigma
\end{align}
where $\otimes^{(k)} \mathcal{E}$ denotes the $k$-fold tensor product of $\mathcal{E}$ with itself and where $\underline{\nabla} (\sigma) : = \underline{\nabla} + \sigma \mathcal{E}$.

When $\underline{\nabla}$ and $\mathcal{E}$ are of the form $\underline{\nabla}  = \underline{\nabla} \phi$ and $\mathcal{E}  = \underline{\nabla} V$ we write
 \begin{align} \notag
 R_{\Phi , V}^{(k)} (\phi) : = R_{\Phi , \mathcal{E} }^{(k)} (\underline{\nabla}).
 \end{align}
 \end{definition}
 
 Note that $R_{\Phi , \mathcal{E}} (\underline{\nabla}) $ is simply the order $k$ Taylor remainder so that:
 \begin{proposition}
 We have 
 \begin{align}\label{TR2}
 \Phi (\underline{\nabla} + \mathcal{E}) - \Phi(\underline{\nabla}) - \left. D\Phi\right|_{\underline{\nabla}} (\mathcal{E}) - \ldots - \frac{1}{k !}\left. D^{(k )} \Phi \right|_{\underline{\nabla}} \left(\otimes^{(k)}\mathcal{E} \right) = R^{(k)}_{\Phi, \mathcal{E}} (\underline{\nabla})
 \end{align}
 \end{proposition}
 \begin{proof}
 Set  $f (\sigma) : = \Phi(\underline{\nabla} (\sigma))$. Recall the integral form of the Taylor remainder theorem implies
 \begin{align*}
 f(1) - f(0) - \ldots - \frac{1}{k!} f^{(k)}(0) = \int_0^1 \frac{(1 - \sigma)^k}{k!} f^{(k + 1)} (\sigma) d \sigma.  
 \end{align*}
 The claim then follows by computing explicitly the derivatives of $f$ in terms of $\Phi$.
 \end{proof}
Since we are interested in immersions, we provide a quantitative statement that well controlled variations of immersions remain immersions.
\begin{proposition} \label{still_regular}
Let $\underline \nabla$ and $\mathcal{E}$ be points in $E$, with $\underline \nabla \in E_0$, satisfying
\begin{align*}
|\mathcal{E}| < \epsilon |\mathfrak{a} (\nabla) | |\nabla|.
\end{align*}
Then for $\epsilon$  sufficiently small, independent of $\nabla$ and $\mathcal{E}$,
\begin{align*}
\left|\mathfrak{a} (\nabla + \mathcal{E}) - \mathfrak{a} (\nabla) \right| < C |\mathcal{E}|/ |\nabla|.
\end{align*}
\end{proposition}
\begin{proof}
The definition of $\mathfrak{a}$ implies there exists $C>0$ independent of $\underline \nabla$ such that 
\begin{align*}
C\mathfrak{a} (\nabla) \leq\frac{|\nabla_1|}{|\nabla_2|} \leq C \mathfrak{a}^{-1} (\nabla).
\end{align*}
This then gives 
\begin{align} \notag
(1 + C  \mathfrak{a}^{2} )|\nabla_2|^2 \leq |\nabla|^2 \leq (1 +  C \mathfrak{a}^{-2})|\nabla_2|^2 \\ \notag
(1 +C  \mathfrak{a}^{2} )|\nabla_1|^2 \leq |\nabla|^2 \leq (1 +  C \mathfrak{a}^{-2})|\nabla_1|^2, 
\end{align}
so that for $0<\epsilon<(4\mathfrak a^2+4C)^{-1/2}$,
\begin{align}\notag
|\mathcal{E}_1|^2 & < \epsilon^2 \mathfrak{a}^2(1  + C\mathfrak{a}^{-2}) |\nabla_1|^2 < \frac{1}{4} |\nabla_1|^2 \\ \notag
|\mathcal{E}_2|^2 & < \frac{1}{4} |\nabla_2|^2.
\end{align}
Set ${\nabla} (\sigma) : ={\nabla} + \sigma \mathcal{E}$. Then  
\begin{align*}
\frac{1}{2}|\nabla_2|  \leq |\nabla_2 (\sigma)| \leq \frac{3}{2}|\nabla_2| \\ \notag
\frac{1}{2}|\nabla_1|  \leq |\nabla_1 (\sigma)| \leq \frac{3}{2}|\nabla_1|.
\end{align*}

It is then straightforward to check that $|\nabla| \left. D \mathfrak{a} \right|_{\nabla (\sigma)}$ is uniformly bounded for $\sigma \in [0, 1]$, so that using \eqref{TR1}, \eqref{TR2}
\begin{align*}
\left|\mathfrak{a} (\nabla + \mathcal{E}) - \mathfrak{a} (\nabla) \right| = \left| R^{(0)}_{\mathfrak{a}, \mathcal{E}}(\nabla)\right| < C |\mathcal{E}|/|\nabla|.
\end{align*}
\end{proof}
Using the previous estimates and the scaling properties of homogeneous functions, we record here an estimate we use with great frequency. In particular, this estimate allows us to control appropriately weighted H\"older estimates on the remainder terms of a homogeneous function by H\"older estimates on the variation field $\mathcal E$.
\begin{proposition} \label{HQEstimates} There exists $\tilde \epsilon>0$ such that if $\underline{\nabla}: D \rightarrow E$ is an immersion and $\mathcal{E}: D \rightarrow E$ is a vector field satisfying
\begin{align} \notag
\|\mathcal{E}: C^{j, \alpha}(D,\mathfrak{a} (\nabla) |\nabla|)\| \leq C(j, \alpha)\tilde \epsilon , \quad \text{ and }\quad \ell_{j, \alpha} (\nabla) : =  \| \nabla: C^{j, \alpha}(D, |\nabla|)\|  < \infty,
\end{align}
then
\begin{align} \notag
\left\| R_{\Phi , \mathcal{E}}^{(k)} (\underline{\nabla}): C^{j, \alpha} (D, |\nabla|^d) \right\| \leq C(\Phi,  \ell_{j, \alpha}, \mathfrak{a},k )\left\| \mathcal{E} : C^{j, \alpha} (D, |\nabla| ) \right\|^{k+1}.
\end{align}
\end{proposition}
\begin{proof}
Since $D^{(k+1)} \Phi $ is homogeneous degree $d - (k+1)$ we can write
\begin{align*}
R_{\Phi , \mathcal{E}}^{(k)} (\underline \nabla)=  |\nabla|^{d - (k+1)}  \int_0^{1} \frac{(1 - \sigma)^k}{k!} \left. D \Phi^{(k + 1)} \right|_{\underline{\nabla}(\sigma)/|\nabla|} (\otimes^{(k + 1)} \mathcal{E}) d \sigma,
\end{align*}
where as before we have set $\underline{\nabla}(\sigma) : = \underline{\nabla} + \sigma \mathcal{E}$. We then have
\begin{align} \label{HQ0}
\left \|  R_{\Phi , \mathcal{E}}^{(k)} (\underline{\nabla}) \right\|_{j, \alpha} & \leq C\left\| |\nabla|\right\|_{j, \alpha}^{d - (k+1)} \left\|  \left.  D^{(k+1)} \Phi \right|_{\underline{\nabla} (\sigma)/| \nabla |} \right\|_{j, \alpha} \left\|  \mathcal{E} \right\|_{j, \alpha}^{k+1}.
\end{align}
 With $C(j, \alpha)\tilde \epsilon\leq\epsilon$ from Proposition \ref{still_regular}, the hypotheses imply $\underline{\nabla} (\sigma)/ |\nabla|$ remains in a fixed compact subset of $E_0$ and
\begin{align*} 
\left\|  \left.  D^{(k+1)} \Phi \right|_{\underline{\nabla} (\sigma)/| \nabla |} \right\|_{j, \alpha}  \leq C (\ell_{j, \alpha}, \mathfrak{a},k).
\end{align*}
Additionally,
\begin{align*}
\left\| |\nabla|\right\|_{j, \alpha}/|\nabla| (s, \theta) \leq C  \ell_{j, \alpha}.
\end{align*}
Dividing both sides of (\ref{HQ0}) by $|\nabla|^{d}$ then gives the claim.
\end{proof}

\section{Logarithmic spirals} \label{LogarithmicSpirals}
Determining embeddedness of the surface in the main theorem requires quantitative estimates that arise from the curvature, torsion, and propogation rate of $\gamma$. In this section we demonstrate that for all non-trivial, anti-symmetric $\RRR$ and for each $\delta>0$ there exists a curve $\gamma$ satisfying \eqref{frame_derivative}. We first consider spirals $\varphi$ determined by three parameters $a,b,c$. We then demonstrate that each $\gamma$ in the family of logarithmic spirals is a scaling of some $\varphi$ where the parameters depend on $\delta, \kappa_0, \tau_0, \xi$ and $\kappa_0, \tau_0$ depend on $\RRR$. Finally, we show that the mapping $M$ is a diffeomorphism from a tubular neighborhood of the $z$-axis onto a logarithmic cone about $\gamma$ where the size of the neighborhood is on the order of $1/\delta$.
\begin{definition} \label{ExpSpirals}
Given constants $a$, $b$ and $c$ we set
\begin{align} \label{logspiralequation}
\varphi[a, b, c] (t) = e^{a  t} \e_r ( c \, t) +\frac ba e^{a  t} \e_z
\end{align}
where we have abbreviated $\e_r(t) : = (\cos (t), \sin (t), 0)$.
\end{definition}

\begin{proposition} \label{ExpSpiralProperties}
The following statements hold:
\begin{itemize}
\item[(1)] The arc length, curvature and torsion of the curve $\varphi[a, b, c](t)$ are given respectively by:
\begin{align} \notag
d s (t) = \sqrt{a^2 + b^2 + c^2}  \, e^{a t}dt, \quad \kappa (t) : =\frac{\sqrt{a^2 c^2 + c^4}}{ a^2 + b^2  + c^2} e^{- a  t}, \quad \tau(t)  =\frac{bc}{a^2 + b^2 + c^2} e^{- a  t}
\end{align}
\item[(2)] $\varphi$ has the following periodicity: 
\begin{align*}
e^{2 \pi a/c} \varphi[a, b, c] (t) = \varphi[a, b, c](t + 2 \pi/ c).
\end{align*}
\end{itemize}
\end{proposition}

\begin{proof}
With $\e_r^\perp = \e_r'$,
\begin{align}\label{phiderivs}
\varphi' (t) = a e^{a  t} \e_r + c  e^{a  t} \e_r^\perp + b  e^{a  t} \e_z \\ \notag
\varphi''(t) = (a^2 - c^2) e^{a  t} \e_r + 2 a c e^{a  t} \e_r^\perp + b a e^{a  t} \e_z.
\end{align}
Then
\begin{align}\notag
|\varphi'| = e^{a  t} (a^2  + b^2  + c^2)^{1/2} :  = A_1 e^{a  t} \\ \notag
|\varphi''| = e^{a  t} (a^4 (1 + b^2/a^2) + 2 a^2 c^2 + c^4)^{1/2} : = A_2 e^{a t}
\end{align}
where the constants $A_1$ an $A_2$ are implicitly defined above. Let $\{T, N, B\}$  be the Frenet frame along $\varphi$. We then have
\begin{align} \notag 
& T = (a \e_r +  c \e_r^\perp + b  \e_z)/A_1, \quad N = \frac{T'}{|T'|} =  \frac{a c  \e_r^\perp -c^2  \e_r}{\sqrt{c^2a^2 +  c^4}} \\ \notag
& B = T \wedge N = \left( a \e_r + c \e_r^\perp + b \e_z \right) \wedge (a c \e_r^\perp - c^2 \e_r)/ (A_1\sqrt{c^2a^2 + c^4}) \\ \notag
&\quad  =  \frac{- ba c \e_r - b c^2 \e_r^\perp +( c a^2 + c^3) \e_z}{A_1 \sqrt{c^4 +c^2 a^2}} \\ \notag
& B' = \frac{- b a c^2 \e_r^\perp + b  c^3 \e_r}{A_1 \sqrt{c^4 + c^2 a^2}}
\end{align}
By calculation, 
\begin{align}\notag
&\kappa (t) = \frac{|T'|}{A_1} e^{- a  t} = \frac{\sqrt{a^2 c^2 + c^4}}{ a^2 + b^2 + c^2} e^{- a  t} \\ \notag
& \frac{dB}{ds}=\frac{B'}{A_1} e^{- a  t } = - \tau N\notag \\
&\tau(t)= \frac{bc}{a^2 +  b^2  + c^2 } e^{- a  t}.\notag
\end{align}
The periodicity follows immediately from the definition of the curve.
\end{proof}

\begin{proposition} \label{LogSpiralsExist}
Given constants $\kappa_0 > 0 $, $\xi,\tau_0\in \Real$, let $\gamma[\kappa_0, \tau_0, \xi] (z) : \mathbb{R} \rightarrow \mathbb{R}^3$ be a smooth curve with
\begin{itemize}
\item[(1)] arc length $ds = e^{\xi  z}dz$.
\item[(2)] curvature $\kappa(z)$ and torsion $\tau(z)$ 
\begin{align} \notag
\kappa(z)  = e^{ - \xi  z}\kappa_0, \quad \tau(z) = e^{ - \xi  z} \tau_0.
\end{align}
\end{itemize}
Then $\gamma$ is unique (up to rigid motions) and determined by 
\begin{equation}\label{gammaform}
 \gamma[\kappa_0, \tau_0, \xi] (z)=   \frac{1}{\sqrt{a_*^2  + b_*^2 + c_*^2}} \varphi[a_*, b_*, c_*](z)
\end{equation}
where 
\begin{align} \notag
a_* = \xi, \quad b_* = \frac{\tau_0}{\kappa_0} \sqrt{\kappa_0^2 + \tau_0^2 + \xi^2} \quad c_* = \sqrt{\kappa_0^2 + \tau_0^2}.
\end{align}

\end{proposition}

\begin{proof}
We determine \eqref{gammaform} explicitly which immediately gives the properties of $\gamma$. The uniqueness then follows from the fundamental theorem of space curves.

From Proposition \ref{ExpSpiralProperties} it follows that the arc length, curvature and torsion for the curve 
\begin{align}\notag
\tilde{\varphi} : = \frac{1}{ \sqrt{a^2  + b^2 + c^2}} \varphi[a, b, c](t)
\end{align}
 are given respectively by
\begin{align} \notag 
  d\tilde{s} & = e^{a  t}dt\\ \notag 
   \tilde{\kappa} & = \sqrt{\frac{c^4 +c^2 a^2}{a^2 + b^2  + c^2}} e^{- a  t} : = \tilde{\kappa}_0 e^{- a  t} \\ \notag
    \tilde{\tau} &  ={\frac{bc}{\sqrt{a^2 + b^2 + c^2}} }e^{- a  t} : = \tilde{\tau_0} e^{- a  t}
    \end{align}
The arc length calculation implies $a_*$ must equal $\xi$. We also note that
\[
\tilde \kappa_0^2 + \tilde \tau_0^2 = c^2
\]which gives the expression for $c_*$. Since
\begin{align}\notag
\frac{\tilde{\kappa_0}^2 (0)} {\tilde{\tau_0}^2 (0) } = \frac{c^4 + c^2 a^2}{ b^2 c^2} = \frac{c^2 + a^2}{b^2}
\end{align} substituting in the values for $a_*$ and $c_*$ gives the expression for $b_*$ since
\begin{align}\notag
b = \frac{\tilde{\tau}_0}{\tilde{\kappa}_0} \sqrt{c^2 + a^2} = \frac{\tilde{\tau}_0}{\tilde{\kappa}_0} \sqrt{\tilde{\kappa}_0^2  + \tilde{\tau}_0^2+ \xi^2}.
\end{align} 
\end{proof}

\begin{proposition} \label{CurveClassification}
Let $\RRR$ be a fixed non-trivial anti-symmetric matrix and $\xi\in \Real$. Then for all $\delta >0$, the following hold:
\begin{itemize}
\item[(1)] There exists a unique smooth curve $\gamma(z): \mathbb R \rightarrow \mathbb R^3$ with a frame $\e$ along $\gamma$ satisfying  (\ref{frame_derivative}). Up to rigid motions this curve has the form $\gamma = \gamma[ \delta \kappa_0, \delta \tau_0,   \delta\xi]$ where 
\[
\kappa_0:=|\RRR \e_3|, \quad \tau_0 = \frac{\langle \e_3 \wedge \RRR^2 \e_3, \RRR \e_3 \rangle}{|\RRR \e_3|^2}.
\]

\item[(2)] Set $\rho_0 : = \sqrt{\kappa_0^2 + \tau_0^2}$. Then for any $\alpha$ satisfying
\begin{align} \notag
0 \leq \alpha <  \frac{ e^{ \pi \xi / \rho_0} - 1}{ e^{ \pi \xi / \rho_0} + 1},
\end{align}
 the map $M$ defined in (\ref{Mdef}) is a diffeomorphism of the tubular neighborhood 
\begin{align} \notag
 T(\alpha) : =  \left\{ (x, y, z) : x^2 + y^2 \leq \frac{\alpha^2}{\delta^2  \xi^2} \frac{\tau_0^2 + \xi^2}{\rho_0^2 + \xi^2}\right\}
 \end{align} onto its image. 
\end{itemize}
\end{proposition}
\begin{proof}
To prove statement $(1)$, fix an initial point and presume $\e_1(0), \e_2(0), \e_3(0)$ corresponds to the standard basis in $\Real^3$. Note that \eqref{frame_derivative} implies
the Frenet frame for $\gamma$ will have the form
\begin{align}\notag 
T = \e_3, \quad N = \frac{\RRR \e_3}{|\RRR \e_3|}, \quad B = \frac{\e_3 \wedge  \RRR \e_3}{|\RRR \e_3|}.
\end{align}
Additionally, the arc length for $\gamma$ will be given by $ds =  e^{\delta \xi  z}dz$.  It then follows that
\begin{align}\notag
\kappa = |T'| e^{- \delta \xi z} ={\delta e^{- \delta \xi z}}{|\RRR \e_3|}, \quad \tau =   B' \cdot N e^{- \delta \xi  z} = \delta \frac{\langle \e_3 \wedge \RRR^2 \e_3, \RRR \e_3 \rangle}{|\RRR \e_3|^2} e^{- \delta \xi  z}
\end{align}
where $\kappa$ and $\tau$ denote the curvature and torsion for $\gamma$, respectively.  The uniqueness result of Proposition \ref{LogSpiralsExist} then implies statement $(1)$.

To prove statement $(2)$, note that  
\begin{align} \label{CC2}
\left. D  M \right|_{(x, y, z)} = 
e^{\delta  \xi z} \left( \e_x^* \otimes \e_1 + \e_y^*\otimes \e_2 +  \e_z^* \otimes  \e_3+ \delta \, x \e_z^* \otimes \e_1   + \delta \, y \e_z^* \otimes \e_2   \right).
\end{align}
Thus, $\det{\left. D M \right|_{(x, y, z)}} = e^{3 \delta \xi  z}$ and $M$ is a local diffeomorphism at each point.  

We now show $M$ is a global diffeomorphism. First note that for each $z=c$, $M(x,y,c)$ is a disk contained in a plane through the origin with normal direction $\e_3(\gamma(c))$. Now suppose $M(x_1, y_1, z_1) = M(x_2, y_2, z_2) = \mathbf p$. Then $\mathbf p \cdot \e_3(\gamma(z_1))= \mathbf p \cdot \e_3(\gamma(z_2))=0$. The definition of $\gamma$ implies $\e_3(\gamma(z)) \cdot \e_z$ is constant and thus $\mathbf p \cdot (\e_3(\gamma(z_1))-\e_3(\gamma(z_2)))=0$ is a condition on radial values (i.e., the values in the directions $\e_x, \e_y$). We immediately conclude $\e_3(\gamma(z_1))=\pm \e_3(\gamma(z_2))$. Thus, \eqref{phiderivs} implies that up to reordering $z_1$ and $z_2$,
\[
z_1-z_2 =  \pi n / c_* =  \pi n / (\delta \rho_0) \text{ for } n \in \mathbb Z^+.
\]
Using \eqref{logspiralequation} and Proposition \ref{LogSpiralsExist}, 
\begin{align}\notag
 |M(0, 0, z)|  = \frac{\sqrt{1 + b_*^2/a_*^2}}{\sqrt{a_*^2 + b_*^2 + c_*^2}} e^{\delta   \xi z}
\end{align}
where 
\begin{align}\notag
a_* = \delta \xi, \quad b_* = \delta \frac{\tau_0}{\kappa_0} \sqrt{\rho_0^2 + \xi^2}, \quad c_* = \delta \rho_0.
\end{align}
From this we get directly that 
\begin{align}\notag
b_*^2/a_*^2 + 1  & =  \frac{1}{ \xi^2 \kappa_0^2} \left( \tau_0^2\rho_0^2 + \tau_0^2\xi^2 + \xi^2  \kappa_0^2 \right) \\ \notag
& = \frac{\rho_0^2}{\xi^2 \kappa_0^2} \left( \tau_0^2 + \xi^2\right), \\ \notag
a_*^2 + b_*^2 + c_*^2 & = \delta^2 \left(\xi^2 + \rho_0^2 + \tau_0^2/\kappa_0^2 (\rho_0^2 + \xi^2) \right)  \\ \notag
& = \delta^2 (\rho_0^2 + \xi^2) (1 + \tau_0^2/\kappa_0^2).
\end{align}
We then conclude that 
\begin{align}\notag
|M(0, 0, z)| &= \frac {e^{\delta \xi z}} {\delta \xi} \sqrt{\frac{\tau_0^2 + \xi^2}{\rho_0^2 + \xi^2}}.
\end{align}
Therefore
\begin{align} \notag
|M(0, 0,  z + n \pi/ (\delta \rho_0)) - M(0, 0, z)| & \geq \left(e^{n \pi \xi/ \rho_0}  - 1\right) |M(0, 0, z)|   \\ \notag
&  \geq \left(e^{n \pi \xi/ \rho_0}  - 1\right)\frac{e^{\delta \xi z} }{ \delta \xi  } \sqrt{ \frac{\tau_0^2 + \xi^2}{\rho_0^2 + \xi^2}} \\ \notag
 | M(x, y, z) - M(0, 0, z) |  & = e^{\delta \xi  z} \sqrt{x^2 + y^2}.
\end{align}
Presuming $\sqrt{x^2+y^2} \leq \frac\alpha { \delta \xi }\sqrt{\frac{\tau_0^2 + \xi^2}{\rho_0^2 + \xi^2}}$ for some $\alpha$ to be determined, we note
\begin{align}\notag
|M(x_1, y_1, z + n \pi/ (\delta \rho_0) ) - M(x_0, y_0, z)| & \geq -   |M(x_1, y_1, z + n \pi/ (\delta \rho_0) ) - M(0, 0, z   + n \pi/ (\delta \rho_0))| \\ \notag 
& \quad-  |M(x_0, y_0, z  ) - M(0, 0, z  )|  \\ \notag
& \quad + |M(0, 0,  z  + n \pi/ (\delta \rho_0)) - M(0, 0, z)|\\ \notag 
& \geq  \left(e^{n \pi \xi/ \rho_0}  - 1\right)\frac{e^{\delta \xi z}}{ \delta \xi} \sqrt{ \frac{\tau_0^2 + \xi^2}{\rho_0^2 + \xi^2}}  -  (1 + e^{n \pi \xi/\rho_0}) \frac {e^{\delta \xi z}\alpha}{ \delta \xi }\sqrt{\frac{\tau_0^2 + \xi^2}{\rho_0^2 + \xi^2}}.
\end{align}
Thus, requiring that 
\begin{align}\notag
 \left(e^{n \pi \xi/ \rho_0}  - 1\right) -  \alpha (1 + e^{n \pi \xi/\rho_0}) > 0 
\end{align}
gives that the points $M(x_1, y_1, z  + n \pi/ (\delta \rho_0))$ and $M(x_0, y_0, z)$ do not intersect.  This completes the proof.
\end{proof}

\section{Geometric quantities on $G$} \label{GeometricQuantitiesOnG}
In this section we record estimates of the relevant geometric data for the immersion $G$. Throughout this section we presume that for a given $\RRR, \xi$, we choose $0<\delta<\delta_0/|\xi|$ such that $\delta(1+|\RRR|+|\xi|) < \tilde \epsilon$. We then define $G$ by a curve $\gamma$ satisfying \eqref{frame_derivative}.

\subsection{The normalized derivatives $\tilde{\nabla} G^{(k)}$}
In order to conveniently estimate geometric quantities on $G$ such as the mean curvature, we want to  normalize $G$ and its derivatives in a way that controls for rotations and dilations, whose effect on quantities such as the mean curvature and unit normal are easily extracted.   
\begin{definition} \label{GNormalized}
 We let $R(\theta)$ be the rotation given below:
\begin{align*}
R( \theta)  : = \e_1 (\theta) \otimes \e_x^* + \e_2 (\theta) \otimes \e_y^* + \e_3 (\theta) \otimes \e_z^*
\end{align*}where $\{\e_x^*, \e_y^*, \e_z^*\}$ is the dual basis in $\Real^3$ to the standard basis. 
We then define the normalized derivatives of $G$ as:
\begin{align*}
\tilde{\nabla}^{(k)} G (s, \theta): = e^{-\delta \xi \theta}R(\theta)\nabla^{(k)} G (s, \theta),
\end{align*}
and we set
\begin{align*}
\tilde{\underline{\nabla}}[ G] : = \left( \tilde{\nabla} G, \tilde{\nabla}^2 G\right).
\end{align*}
\end{definition}

\begin{proposition} \label{normalized_derivative_properties}
The following statements hold:
\begin{enumerate}
\item \label{1}The normalized derivatives $\tilde{\nabla} G(s, \theta)$ are $2 \pi$-periodic in $\theta$.
\item \label{2}Letting $\tilde \nu_G:=\nu (\tilde{\nabla} [G]) $,
\begin{align} \notag
\nu_G =\nu(\nabla G) = R^{-1}(\theta)\left( \tilde \nu_G \right).
\end{align} 
\item \label{3}For $H_G:=H(\underline {\nabla} G)$, 
\begin{align} \notag
H_G = e^{-\delta \xi \theta} H (\tilde{\underline{\nabla}} [G]).
\end{align}
\end{enumerate}
\begin{proof}
Item \eqref{1} follows directly from the definition of $F$ in (\ref{Fdef}). Items \eqref{2} and \eqref{3} follow from the fact that $\nu$ and $H$ are homogeneous degree $0$ and $-1$ quantities, respectively, and their behavior under rotations and dilations. 
\end{proof}
\end{proposition}
\subsection{Comparing immersions on spirals with straight lines}
In \cite{BK}, we considered immersions of the form:
\begin{align} \label{G_0Def}
G_{0  \delta}(s, \theta) =\left( e^{\delta  \theta} \sin (\theta)\sinh(s),  e^{\delta  \theta} \cos(\theta)\sinh(s), \frac{1}{\delta} e^{\delta  \theta}  \right).
\end{align}

\begin{theorem}[Theorem 1 from \cite{BK}] \label{BK1Theorem}
There are constants $ \epsilon_0, \delta_0  > 0$ sufficiently small so that for any $0<\delta < \delta_{0}$, there is a function $u_{0 \delta} (s): [-\epsilon_0 \delta^{-1/4}, \epsilon_0 \delta^{-1/4}] \rightarrow \Real$ such that:
\begin{enumerate}
\item The normal graph over $G_{0 \delta}$ by the function $w_{0\delta}(s, \theta) :  = e^{\delta \xi \theta} u_{0\delta} (s)$ is an embedded minimal surface with boundary.
\item $u_{0\delta}(s)$ is an odd function.
\item There is a constant  $C > 0 $ sufficiently large so that $u_{0\delta}$ satisfies the estimate
\begin{align}\notag
\| u_{0\delta}: C^{j, \alpha} ([0, \epsilon_0 \delta^{-1/4}], s^2) \| \leq C \delta.
\end{align}
\end{enumerate}
\end{theorem}Notice that for any $\xi\neq0$, we can apply this theorem to immersions $G_{0\, \delta \xi}$ as long as $0<\delta|\xi|<\delta_0$ and all domain bounds and estimates then have $\delta$ replaced by $\delta |\xi|$. Moreover, as $\xi \to 0$, an appropriate translation of $G_{0\, \delta \xi}$ converges to the helicoidal embedding by $F$. Thus, when $\xi =0$, it will be natural to replace estimates for $G_{0 \, 0}$ by estimates for $F$. 

Throughout, we will take the liberty of suppressing the dependence of these immersions and maps on $\delta, \xi$ from the notation and instead write
\begin{align}\notag
G_0 : = G_{0 \, \delta \xi}, \quad u_0 := u_{0 \, \delta \xi}, \quad w_0 := w_{0 \, \delta \xi}.
\end{align}

As a first approximation to our solution we  wish to compare the geometry of $G$ with that of $G_0$. To do this we  normalize the derivatives of $G_0$ by taking
\begin{equation} \label{G_0Normalized}
\tilde{\nabla} G_0^{(k)} (s, \theta) : = e^{- \delta  \xi \theta}  \nabla G_0 (s, \theta).
\end{equation}

\begin{lemma} \label{close_to_a_straight_line} For all $k \in \mathbb Z^+$ there exists $C$ independent of $k$ so that
\begin{align*}
\left| \tilde{\nabla}^{(k) }G (s, \theta)- \tilde{\nabla}^{(k)} G_0(s, \theta)\right| < C \delta |\RRR| \cosh (s).
\end{align*}
\end{lemma}
\begin{proof} Notice first that $\delta|\RRR| <1$. By definition, 
\[
R(\theta) \frac{\partial^k}{\partial s^k} G = \frac{\partial^k}{\partial s^k} G_0
\]so any differentiation only in $s$ will vanish in the difference. Let $\e_r:=\sin \theta \e_1 + \cos \theta \e_2$ and $\e_r^\perp:= \cos \theta \e_1 - \sin \theta \e_2$. Then
\begin{align*}
(G_0)_\theta &= e^{\delta \xi \theta} R(\theta)\e_3+e^{\delta \xi \theta} \sinh(s) R(\theta)\e_r^\perp + \delta \xi e^{\delta \xi \theta} \sinh (s)R(\theta)\e_r\\
 G_\theta &= e^{\delta \xi \theta} \e_3  +e^{\delta \xi \theta} \sinh(s) \e_r^\perp + \delta \xi e^{\delta \xi \theta} \sinh(s) \e_r  + \delta e^{\delta \xi \theta} \sinh(s) \RRR\e_r
 \end{align*}and thus
 \[
 e^{-\delta \xi \theta} \left(R(\theta)G_\theta - (G_0)_\theta \right)= \delta \sinh(s) R(\theta) (\RRR \e_r).
 \]
 This immediately proves the estimates for $k=1$. The higher order estimates follow inductively since $0<\delta |\RRR| < 1$.
\end{proof}
From Lemma \ref{close_to_a_straight_line}, we can obtain a good estimate for the mean curvature of $G_{w_0}$, the normal graph over $G$ by the function $w_0$. Since $(G_{0})_{  w_0} $ is a minimal surface, we expect that the failure of  $G_{w_0}$ to be minimal is controlled by the geometry of the modeling curve for $G$ and the scale $\delta$. We first need to compare the unit normal field along along the immersions. 

To demonstrate that we may use Proposition \ref{HQEstimates} to estimate geometric quantities of small graphs over $G$ we record the following lemma.
\begin{lemma} \label{small_graphs_give_regular_immersions}
The following statements hold:
\begin{itemize}
\item[(1)] There exists $C>0$ independent of $G$ such that 
\begin{align} \notag
C^{-1}\cosh(s) \leq |\tilde{\nabla} [G]| \leq C \cosh(s)
\end{align}
\item[(2)] Recalling the definition of $\ell_{j, \alpha}$ from Proposition \ref{HQEstimates},
\begin{align} \notag
\ell_{j, \alpha} (\tilde{\underline{\nabla}} [G]) \leq C(j, \alpha)
\end{align}
\item[(3)] There exists $C>0$ independent of $G$ such that 
\begin{align}\notag
1 - [\mathfrak{a}(\tilde{\nabla} [G])]  \leq C \delta \left( |\RRR| + |\xi| \right)
\end{align}
\end{itemize}
\end{lemma}
\begin{proof}
Items $(1)$ and $(2)$ are direct consequences of Lemma \ref{close_to_a_straight_line} and the corresponding estimates for $F$. To prove item $(3)$, first note that
\[
|\tilde \nabla^{(k)} G_0 -\nabla^{(k)} F| \leq C \delta |\xi| \cosh(s).
\] The result now follows from the triangle inequality, Lemma \ref{close_to_a_straight_line}, Proposition \ref{still_regular} and the fact that $\mathfrak{a} [F] = 1$.
\end{proof}

\begin{lemma} \label{unit_normal_difference}
For any $j \in \mathbb Z^+$
\begin{align*}
\left\|\tilde \nu_G (s, \theta) - {\nu}_{G_0}(s, \theta): C^{j, \alpha} (\Real^2) \right\| < C (j, \alpha)  |\RRR| \delta.
\end{align*}
\end{lemma}
\begin{proof}First note that $\delta(1 +  |\RRR|+\xi) < \tilde \epsilon$ of Proposition \ref{HQEstimates}. Thus for $\nu$, we can apply Proposition \ref{HQEstimates} with $k=0$ and $d=0$. The result then follows by Lemma \ref{close_to_a_straight_line} and Proposition \ref{HQEstimates} with $\underline{\nabla} = \tilde{\underline{\nabla} }G $ and $\mathcal{E}  = \tilde{ \underline{\nabla} }G - \tilde{\underline{\nabla} }G_0$.
\end{proof}

We will exploit the periodicity of the helicoid and consider graphs over $G=M \circ F$ that are periodic in $\theta$. To that end, we define the appropriate quotient space of $\Real^2$. 
\begin{definition}
Let $\Omega$ be the quotient of $\Real^2$ by the translation $(s, \theta) \mapsto (s, \theta + 2 \pi)$.
\end{definition}

Given an immersion $G$, we look for functions $u:\Omega \to \Real$ such that $G+ u\nu_G$ is minimal. Because of the homogeneity of $H$ and its invariance under rotations, we consider variation fields of the following form.
\begin{definition} \label{FunctionCommutatorDef}
Given  a function $u: \Omega \rightarrow \Real$, we set
\begin{align}
\mathcal{E}_{\delta } [u] : = e^{- \delta  \xi \theta}  R(\theta)\underline{\nabla} (e^{\delta \xi   \theta}u \nu_G).
\end{align}
\end{definition}
The self-similarity of $\gamma$ allows us to consider the mean curvature up to the natural localized rotation and dilation of $G$ by $e^{-\delta\xi\theta}R(\theta)$. To that end, we define the map $Q$. 
\begin{definition}
The map $Q [u]: C^2 (D) \rightarrow C^0 (D)$ is  given as follows:
\begin{align}\label{Qdef}
Q[u] &: = e^{\delta  \xi \theta} \cosh^2(s) H(\underline{\nabla} [G ]+ \underline{\nabla} [e^{\delta \xi \theta}( u + u_0) \nu_G]) \\ \notag
& = \cosh^2 (s) H(\tilde{\underline{\nabla}} [G] + \mathcal{E}_{\delta} [u + u_0]).
\end{align}
\end{definition}
An important consequence of the definition is that $Q$ preserves the periodicity property.
\begin{lemma}\label{Qlemma}
$Q$ maps the space $C^2 (\Omega)$ into $C^0 (\Omega)$.
\end{lemma}
\begin{proof}

Since we have already verified in Proposition \ref{normalized_derivative_properties}, item $(1)$, that $\tilde{\underline{\nabla}} G$ is periodic, the preservation of periodicity follows once we verify that $\mathcal{E}_{\delta} (u)$ maps periodic functions to periodic functions. Item $(1)$ of Proposition \ref{normalized_derivative_properties} also implies $\tilde \nu_G$ is $2\pi$ periodic in $\theta$ and thus, all derivatives of $\tilde \nu_G$ are $2\pi$ periodic in $\theta$. By direct calculation, it is enough to show that $R(\theta) \partial_\alpha \nu_G$ is $2\pi$ periodic for $\alpha = s, \theta, ss, s\theta, \theta \theta$. 

First, note that the vector $R(\theta)\e_i$ is independent of $\theta$. Since $\RRR$ is fixed, and
\begin{equation}\label{theta_ind_eq}
R'(\theta)\e_i = (R(\theta) \e_i)' -R(\theta) \delta \RRR \e_i=-R(\theta) \delta \RRR \e_i,
\end{equation}we note that $R'(\theta)\e_i$ is independent of $\theta$. A similar calculation with second derivatives immediately implies $R''(\theta) \e_i$ is also $\theta$ independent. Since we will need an estimate on $|R''(\theta)|$ later, we record here
\begin{equation}\label{theta_ind_eq_2}
0=\left(R(\theta)\e_i\right)''=R''(\theta) \e_i + 2R'(\theta) \e_i' + R(\theta)\e_i''=\left(R''(\theta) + 2R'(\theta)\delta\RRR + R(\theta)\delta^2\RRR^2\right)\e_i
\end{equation}Now, suppose $\nu_G = \sum_i \alpha_i \e_i$. Then since $\tilde \nu_G = R(\theta) \nu_G$, the $\alpha_i$ are all $2\pi$ periodic in $\theta$. Moreover, one quickly calculates
\begin{align*}
\left(\tilde \nu_G\right)_s - R(\theta)\left(\nu_G\right)_s &= \left(\tilde \nu_G\right)_{ss} - R(\theta)\left(\nu_G\right)_{ss}=0\\
\left(\tilde \nu_G\right)_\theta - R(\theta)\left(\nu_G\right)_\theta&= \sum_i \alpha_i R'(\theta)\e_i\\
\left(\tilde \nu_G\right)_{s\theta} - R(\theta)\left(\nu_G\right)_{s\theta}&= \sum_i (\alpha_i)_s R'(\theta)\e_i\\
\left(\tilde \nu_G\right)_{\theta\theta} - R(\theta)\left(\nu_G\right)_{\theta\theta}&= \sum_i\left( 2(\alpha_i)_\theta R'(\theta) \e_i+ 2 \alpha_iR'(\theta) \delta \RRR \e_i+ \alpha_i R''(\theta) \delta \RRR \e_i\right).
\end{align*}Since all terms on the right are $2\pi$ periodic in $\theta$ and all derivatives of $\tilde \nu_G$ are $2\pi$ periodic in $\theta$, $Q$ preserves periodicity.

\end{proof}
We use the estimates of Lemmas \ref{small_graphs_give_regular_immersions}, \ref{unit_normal_difference} , and \ref{Qlemma} to determine estimates for the mean curvature of $G_{w_0}$ which by definition correspond to estimates on $Q[0]$. This bound will appear again in the fixed point argument of Section \ref{FindingExactSolutions}.
\begin{proposition} \label{normalized_problem_properties}
For $\Omega^*:= \Omega \cap \{s: |s| \leq \epsilon (\delta \xi)^{-1/4}\}$,
\begin{align} \label{NPP1}
\left\| Q [0]: C^{j, \alpha} (\Omega^*,  \cosh (s)) \right\| <  C (j, \alpha) \delta|\RRR|.
\end{align}
\end{proposition}

\begin{proof}
Define $\mathcal{E}_{0  \delta} [u]  : = e^{- \delta   \xi \theta} \underline{\nabla} (e^{\delta \xi  \theta}u \nu_{G_0})
$. To prove the norm bound, we first write
\begin{align}\notag
Q [0] & = \cosh^2(s) H(\tilde{\underline{\nabla}} [G] + \mathcal{E}_{\delta} [u_0]) \\ \notag
& =  \cosh^2(s)\left( H(\tilde{\underline{\nabla}} [G] + \mathcal{E}_{\delta} [u_0]) - H(\tilde{\underline{\nabla}} G_0 + \mathcal{E}_{0  \delta} [u_0]) \right) \\ \notag
& =- \cosh^2(s) R^{(0)}_{H, \underline{\nabla}} (\mathcal{E}),
\end{align}
where $\underline{\nabla} : = \tilde{\underline{\nabla} }[G] +  \mathcal{E}_{\delta} [u_0]$ and $ \mathcal{E} =  \tilde{\underline{\nabla}} [G_0 - G] + \left( \mathcal{E}_{0\delta} [u_0] - \mathcal{E}_{  \delta} [u_0] \right) : = T_1 + T_2$. Recall that $T_1$ has been estimated in Lemma \ref{close_to_a_straight_line}:
\begin{align} \label{NPP5}
\| T_1\|_{j, \alpha}  \leq C (j, \alpha) \delta |\RRR| \cosh (s)
\end{align}
so we focus on the term
\begin{align*}
T_2 : = \mathcal{E}_{0\delta} [u_0] - \mathcal{E}_{  \delta} [u_0]=  e^{- \delta  \xi \theta} \underline{\nabla} (e^{\delta  \xi \theta} u_0   \nu_{G_0})-e^{- \delta \xi \theta} R(\theta) \underline{\nabla} (e^{\delta   \xi \theta} u_0   \nu_G).
\end{align*}
By direct computation, we determine the components of $T_2$. Projected onto $E^{(1)}$,
\begin{align*}
\left( (u_0)_s\left(R(\theta)\nu_G-\nu_{G_0}\right) + u_0\left( R(\theta)(\nu_G)_s - (\nu_{G_0})_s\right) , \delta\xi u_0\left( R(\theta)\nu_G-\nu_{G_0}\right) + u_0 \left(R(\theta)(\nu_G)_\theta - (\nu_{G_0})_\theta \right)\right)
\end{align*}and onto $E^{(2)}$
\begin{align*}
&\left((u_0)_{ss}\left(R(\theta)\nu_G-\nu_{G_0}\right)+ 2(u_0)_s \left( R(\theta)(\nu_G)_s - (\nu_{G_0})_s\right)+ u_0 \left( R(\theta)(\nu_G)_{ss} - (\nu_{G_0})_{ss}\right), \right.\\
&\delta\xi (u_0)_s \left(R(\theta)\nu_G-\nu_{G_0}\right)+ \delta\xi u_0\left( R(\theta)(\nu_G)_s - (\nu_{G_0})_s\right)+ (u_0)_s \left(R(\theta)(\nu_G)_\theta - (\nu_{G_0})_\theta \right)+ u_0 \left(R(\theta)(\nu_G)_{s\theta} - (\nu_{G_0})_{s\theta} \right),\\
&\left. (\delta\xi)^2u_0 \left(R(\theta)\nu_G-\nu_{G_0}\right)+ 2\delta\xi u_0 \left(R(\theta)(\nu_G)_\theta - (\nu_{G_0})_\theta \right)+ u_0 \left(R(\theta)(\nu_G)_{\theta\theta} - (\nu_{G_0})_{\theta\theta} \right)
\right).
\end{align*}
From Lemma \ref{unit_normal_difference}, since $\tilde \nu_G = R(\theta) \nu_G$ 
\[
\| R(\theta)\nu_G-\nu_{G_0}:C^0\| \leq C\delta|\RRR|.
\]
For the projection onto $E^{(1)}$, we use Lemma \ref{unit_normal_difference}, the triangle inequality, and \eqref{theta_ind_eq} to get, for $i \in \{s, \theta\}$,
\begin{equation}\label{NPP3}
|R(\theta)(\nu_G)_i-(\nu_{G_0})_i| \leq C \delta|\RRR|.
\end{equation}For the projection onto $E^{(2)}$, we observe first that for $i, j \in \{s,\theta\}$,
\[
\left(\tilde \nu_G\right)_{ij} - R(\theta) \left(\nu_G\right)_{ij} = \left(R(\theta)\right)_{ij} \nu_G + \left(R(\theta)\right)_i \left(\nu_G\right)_j + \left(R(\theta)\right)_j \left(\nu_G\right)_i.
\]Appealing to Lemma \ref{unit_normal_difference}, the triangle inequality, \eqref{theta_ind_eq} again coupled with \eqref{theta_ind_eq_2}, we observe that, since $\delta |\RRR| <1$,
\begin{equation}\label{NPP4}
|R(\theta)(\nu_G)_{ij}-(\nu_{G_0})_{ij}| \leq C \delta|\RRR|.
\end{equation}

Combining (\ref{NPP5}),  (\ref{NPP3}) and (\ref{NPP4}), and noting further that $0 \leq \delta |\xi| <1$,
\begin{align*}
\| \mathcal{E}\|_{j, \alpha}& \leq C(j, \alpha) \delta|\RRR|  \left( \cosh(s) + \| u_0\|_{j + 2, \alpha}  \right) \\ \notag
& \leq C (j, \alpha)\delta|\RRR| \left(\cosh(s) + s^{j+2} \right).
\end{align*}
Using the estimates from Lemma \ref{small_graphs_give_regular_immersions}, we apply Proposition \ref{HQEstimates} with $H= \Phi$ and $d = -1$.
\end{proof}

\section{The linear problem} \label{TheLinearProblem}
The goal of this section is to prove Proposition \ref{FlatTotalInverse} which shows the linear operator $\cosh^2(s)\mathcal L_F$ is invertible in appropriately weighted H\"older spaces, modulo a two dimensional space of exponentially growing functions. These weighted spaces will be defined on subsets of $\Omega$.
\begin{definition} \label{X_spaces_def}
Set
\begin{align}  \label{LambdaDef}
\Lambda : =  \Omega \cap \{ |s | \leq \arccosh( \ell) \}
\end{align}
where $\ell  \in   (2 ,  \cosh(\epsilon_0 (\delta|\xi|)^{-1/ 4}))$ is a constant to be determined and where $\epsilon_0$ is as in the statement of Theorem  \ref{BK1Theorem}. The upper bound for $\ell$ is the maximum scale on which the functions $u_0$ of \cite{BK} are defined. The main theorem does not allow such a generous upper bound for $\ell$, though $\ell$ should be considered as a large constant. The weighted spaces on which we solve the linear problem now take the following form.
\end{definition}

\begin{definition} \label{XSpaces}
Let $\mathcal{X}^{k}$, $k = 0, 2$ be the space of  functions $f(s, \theta)$ in $C^{k, 3/4}_{loc} (\Lambda) $ such that
\begin{align}
\| f : \mathcal{X}^{k}\| :  = \| f : C^{k, 3/4} ( \Lambda, \cosh^{3/4}(s))\| < \infty.
\end{align}

\end{definition}

Note that the spaces $\mathcal{X}^{k}$ are Banach spaces with norm $\| - : \mathcal{X}^{k}\|$.
Our main result will follow from the fact that linearized problem on the surface $F$ is invertible in the spaces $\mathcal{X}^{k}$, and the fact that we can treat the linearized problem on $G$ as a perturbation of the linearized problem on $F$.

For a locally class $C^2$ immersion $\phi  (s, \theta): D \rightarrow \mathbb{R}^3$, recall $\mathcal{L}_\phi$ denotes the \emph{stability operator}:
\begin{align}
\mathcal{L}_\phi : = \Delta_\phi + |A_\phi|^2
\end{align}
where $\Delta_\phi$ and $|A_\phi|^2$ denote the Laplace operator and the squared norm of the second fundamental form of $\phi$, respectively.


\subsection{Prescribing the kernel content of the error term}
In the spirit of Kapouleas, we solve the linear problem on $F$ by modifying an inhomogeneous $E \in \mathcal X^0$ so the modified function is $L^2$ orthogonal to the obstructions to invertibility. As these obstructions arise because of properties of $F$, we first record some relevant quantities.

Let $g_F$, $\nu_F$ and $A_F$ be the metric, the unit normal, and the second fundamental for $F$, respectively. Then
\begin{align}
g_F (s, \theta)  & = \cosh^2 (s) (ds^2 + d\theta^2) \notag\\ \label{FGeomData}
\nu_F (s, \theta) & = - \cosh^{-1} (s) \cos (\theta) \e_x + \cosh^{-1} (s) \sin(\theta) \e_y + \tanh (s) \e_z \\ \notag
A_F (s, \theta) & = - 2 ds d\theta\\
\cosh^2(s) \mathcal{L}_F &=  \Delta_{\Omega} + 2 \cosh^{-2} (s).\notag
\end{align}

Note that if $u: \Omega \to \Real$ then the periodicity of $F$ and $\nu_F$ imply that $u$ can be extended to a graph over the full helicoid. Moreover, to understand the behavior of $\mathcal L_F u$, it is enough to understand its behavior on a fundamental domain of the helicoid. That is, we can consider the problem only on $\Omega$ rather than on all of $\mathbb R^2$. 
In the same spirit, we can analyze the Gauss map $\nu_F: F \rightarrow \mathbb S^2$ on $\Omega$. On this subdomain, $\nu_F$ is a conformal diffeomorphism with conformal factor $|A_F|^2/2$ onto the punctured sphere $\mathbb S^2 \setminus \{ (0, 0, \pm 1)\}$. 

We now define our perturbing functions.
\begin{definition}
Let $\psi(s)$ be the cutoff function
$\psi(s) := \psi[1 , 2 ] (|s|)$ and set
\begin{align*}
u_x (s, \theta) = \frac{1}{4 \pi}\psi (s)\cos (\theta) \cosh(s), \quad u_y (s, \theta) = \frac{1}{4 \pi}\psi (s) \sin (\theta) \cosh (s), \quad u_z (s, \theta) =\frac 1{4\pi} \psi (s) \, |s|.
\end{align*}
The linear changes in the mean curvature due to adding the graphs of $u_x$, $u_y$ and $u_z$ are then:
\begin{align*}
w_x : = \cosh^2(s)\mathcal{L}_{F} u_x, \quad w_y : = \cosh^2(s)\mathcal{L}_{F} u_y, \quad w_z : = \cosh^2(s)\mathcal{L}_{F} u_z.
\end{align*}
\end{definition}

\begin{definition} \label{KappaDefs}
Let $\underline{\kappa}$ be the space of bounded functions $\kappa: \Omega \to \Real$ such that $\mathcal{L}_F \kappa =0$. Then standard theory implies $\underline{\kappa}$ is spanned by the functions
\begin{align} \notag
\kappa_x = \cos (\theta) \cosh^{-1} (s), \quad \kappa_y = \sin (\theta) \cosh^{-1} (s), \quad \kappa_z = \tanh(s).
\end{align}
\end{definition}

The functions $\kappa_x$, $\kappa_y$ and $\kappa_z$ are (up to sign) the $x$, $y$, and $z$ components of the unit normal for $F$. Let $\bar{\kappa}_x$, $\bar{\kappa}_y$ and $\bar{\kappa}_z$  be the lift to sphere of $\kappa_x$, $\kappa_y$ and $\kappa_z$, respectively, under the Gauss map of $F$. Then 
\begin{align} \notag
\bar{\kappa}_x = \bar{x}, \quad \bar{\kappa}_y = \bar{y}, \quad \bar{\kappa}_z = \bar{z}
\end{align} where $\bar{x}$, $\bar{y}$ and $\bar{z}$ denote the restriction to $\mathbb S^2$ of the ambient coordinate functions $x$, $y$ and $z$.

\begin{proposition} \label{KernelStuff}
For $\kappa w \in \{ \kappa_x w_x, \kappa_y w_y, \kappa_z w_z\}$, 
\begin{align*}
\int_\Omega \kappa \,  w \, d \mu_\Omega  = 1.
\end{align*}
\end{proposition}
\begin{proof}
Let $N$ be a large constant and set $D : = D_N = \Omega \cap \{ |s| \leq N\}$ and $\partial = \partial_D : = F(|s| = N)$.  Then
\begin{align} \label{KS1}
\int_D \kappa w \,  d\mu_\Omega = \int_D \kappa  \mathcal{L}_F u d\mu_F = \int_{\partial D} \kappa \nabla^F_\eta u \, d \mu_\partial  - u \nabla^F_\eta \kappa d \mu_{\partial }
\end{align}
where $\nabla^F$ denotes the surface gradient on $F$ and where $\eta$ is the outward pointing conormal at $\partial D$. We do not write in the $\cosh^2(s)$ term in the second integral since $\cosh^2(s)d\mu_\Omega = d\mu_F$. Notice that $\partial D$ consists of four components, though only two will be needed in the calculation. 

Observe that for $\theta = \pm \pi$, $\partial_\theta \kappa_i=\partial_\theta u_i=0$ for $i \in \{x, y, z\}$. Thus, we are only concerned with the boundary components $s = \pm N$. From \eqref{FGeomData}, $\nabla^F_\eta : = \pm \cosh^{-1} (N) \partial_s$ for $\pm s > 0$ and $d \mu_{\partial} = \cosh(N) d\theta$. Thus
\begin{align*}
4\pi \int_D \kappa_x w_x \,  d\mu_\Omega =& \int_{-\pi}^{\pi}\cos^2(\theta) \tanh(N) d\theta - \int_{-\pi}^{\pi} - \cos^2(\theta)\tanh(N) d\theta\\
&+ \int_{-\pi}^{\pi} \cos^2(\theta)(-\tanh(-N)) d\theta-\int_{-\pi}^{\pi}\cos^2(\theta) \tanh(-N) d\theta\\
=& 4\tanh(N) \int_{-\pi}^{\pi} \cos^2(\theta) d\theta \rightarrow 4\pi.
\end{align*}
The same estimate follows for $\int_{\Omega} \kappa_y w_y d \mu_F$. For $w = w_z$ and $\kappa = \kappa_z$, 
\begin{align*}
\nabla_\eta^F \kappa_z = \cosh^{-3} (N).
\end{align*}
In this case, the second term on the right hand side of  (\ref{KS1}) converges to zero as $N$ goes to infinity, and we note that 
\begin{align}
4\pi\int_D \kappa_z w_z d\mu_\Omega = \int_{-\pi}^{\pi} \tanh(N) d\theta - \int_{-\pi}^{\pi} \tanh(-N) d\theta \rightarrow 4 \pi.
\end{align}
\end{proof}

\subsection{Inverting the stability operator in the spaces $\mathcal{X}^k$ modulo $\underline{\kappa}$}
With our modifying functions in hand, we first demonstrate we can invert the operator $\cosh^2(s) \mathcal L_F$ over the space of functions orthogonal to $\underline \kappa$. 
\begin{definition} \label{XPerpSpaces}
We let $\mathcal{X}^{k \, \perp}  \subset \mathcal{X}^k$ be the space of functions orthogonal to $\underline{\kappa}$ on $\Omega$. That is,  a function $f$ is in $\mathcal{X}^{k \, \perp }$ if and only if $f$ is in $\mathcal{X}^{k}$ and for all $\kappa \in \underline \kappa$,
\begin{align*}
\int_\Omega f \,  \kappa \, d \mu_{\Omega} = 0.
\end{align*} 
\end{definition}

We find it convenient to solve the linear problem for inhomogeneous $E$ with $E|_{\partial \Lambda} =0$ as a few technical arguments are made easier. As a trade-off, we have to be a bit more careful in applying the linear theory to the fixed point argument in Section \ref{FindingExactSolutions}.

\begin{proposition} \label{flat_inverse}
 Let $\mathcal{X}^{0 \, \perp}_0  \subset \mathcal{X}^{0 \, \perp}$ be the subspace of functions that vanish on $\partial \Lambda$. Then there is a bounded  linear map 
\begin{align*}
\mathcal{R}_F^\perp : \mathcal{X}^{0 \, \perp}_0 \rightarrow \mathcal{X}^{2} 
\end{align*}
such that for $E \in \mathcal{X}^{0\, \perp}_0$,
\begin{align*}
\cosh^{2}(s) \mathcal{L}_F \mathcal R_F^\perp [E] = E.
\end{align*}
\end{proposition}

The proof of Proposition \ref{flat_inverse} follows from three main steps. We first show $\Delta_\Omega$ can be inverted over the space of H\"older functions with a convenient averaging property. 
In the second step, we show there is a sufficiently large constant $s_0$  so that on $\Lambda \cap \{ s > s_0\}$  the operator $\cosh^2(s)\mathcal{L}_F $ can be solved as a perturbation of the Laplace operator in the spaces $\mathcal{X}^{k}$. In this way we reduce to the case that the error term has large but fixed compact support. In the third step, we solve for the error term by lifting the problem to the sphere using the Gauss map of $\nu_F$, where the problem reduces to an eigenvalue problem for the stability operator on the sphere.

\begin{definition} \label{ZeroMeridianAverageSpaces}
Let $\mathring{C}^{k, \alpha}_{loc}(\Omega)$ denote the space of functions $E \in C^{k,\alpha}(\Omega)$ satisfying the condition
\begin{align} \label{MeridianCondition}
\int_{- \pi}^\pi E(s, \theta) d \theta = 0
\end{align}
for all $s$. A function satisfying (\ref{MeridianCondition}) is said to have {\bf zero average along meridians}. Given a positive weight function $f$, we then denote
\begin{align*}
\mathring{C}^{k, \alpha} (\Omega, f) : = C^{k, \alpha} (\Omega, f)  \cap  \mathring{C}^{k, \alpha}_{loc}.
\end{align*}
\end{definition}

We will prove the invertibility of $\Delta_\Omega$ over $\mathring C^{0, \alpha}(\Lambda)$ by first considering the invertibility of the local problem.
\begin{lemma} \label{LaplaceCPTSPTInverse}
Let $A_0 \subset \Lambda$ be the annulus
\begin{align}\notag
A_0 : = \Omega \cap \{ |s| \leq 5/8 \}
\end{align}
and let $\mathring{C}_0^{k, \alpha} (A_0)$ be the space of $\mathring C^{k, \alpha}$ functions on $\Omega$ with support on $A_0$.

Given  a compact set $K$ containing $A_0$,  there is a bounded linear map
\begin{align}\notag
\mathring{\mathcal{R}}_0[- ]: \mathring{C}^{0, \alpha}_0 (A_0) \rightarrow C^{2, \alpha}(\Omega) \cap \mathring{C}^{k} (\Omega \setminus K, \cosh^{-1} (s)) 
\end{align}
such that
\begin{align}\notag
\Delta_\Omega \mathring{\mathcal{R}}_0[E] = E.
\end{align}
\end{lemma}
\begin{proof}
Lemma \ref{LaplaceCPTSPTInverse} can be established several ways. We choose the following approach. Let $\Omega_L$ be the domain 
\begin{align} \label{OmegaLDef}
\Omega_L : = \Omega\cap \{ |s| \leq L\}.
\end{align}
In other words $\Omega_L$ is the just the a flat cylinder of length $2 L$ centered at the meridian $\{s = 0\}$. Standard elliptic theory gives the existence of  functions $u_L \in C^{2, \alpha}_{loc} (\Omega_L)$ satisfying:
\begin{align}\label {LCPTSPTI1}
\Delta u_L = E, \quad u_L (\theta, \pm L) = 0.
\end{align}
It is then direct to verify that the functions $u_L$ satisfy:
\begin{align} \label{LCPTSPTI2}
\int_{- \pi}^\pi u (s, \theta) d \theta = 0.
\end{align}
To see this, we integrate both sides of the first equality in  (\ref{LCPTSPTI1}) in $\theta$ to obtain
\begin{align} \notag
\left(\int_{- \pi}^\pi u (s, \theta) d \theta\right)_{s  s} = 0.
\end{align}
The boundary conditions  in (\ref{LCPTSPTI1}) then imply (\ref{LCPTSPTI2}). From this and the  Poincare inequality for $K$ \begin{align} \label{LCPTSPT4}
\sup_{K} |u_L| \leq C(K) \| E : C^{0, \alpha} (K) \| \leq C(K) \| E : C^{0, \alpha} (A_0)\| .
\end{align}
We then set
\begin{align} \label{LCPTSPT5} 
\mathring{\mathcal{R}}_0  [E]=u_\infty (x, s) : = \lim_{L \rightarrow \infty} u_L(x, s).
\end{align}
From (\ref{LCPTSPT4}), the limit in \eqref{LCPTSPT5} exists in $C^{2, \alpha}$ and by continuity $u_\infty$ solves
\begin{align*}
\Delta_{\Omega} u_\infty = E, \quad \int_{- \pi}^\pi u_\infty(s, \theta) d\theta  = 0.
\end{align*}
The exponential decay of $u_\infty$ in both the positive and negative $s$ directions then follows directly.  
\end{proof}

We are now ready to prove the invertibility of $\Delta_\Omega$ over $\mathring{C}^{0, \alpha}_0 (\Lambda)$. Following standard arguments, we will sum up the solutions to the local problem found by Lemma \ref{LaplaceCPTSPTInverse} and show that the estimates are sufficient to establish convergence.
\begin{proposition} \label{LaplaceInverse}
Given $\rho \in (-1 , 1) \setminus \{ 0 \}$, $k \geq 0$, $\alpha \in (0, 1)$ and a compact set $K$ containing $\Lambda$, there is a bounded linear map 
\begin{align}\notag
\mathcal{R}_{\rho \,\alpha} [-]: \mathring{C}^{0, \alpha}_0 (\Lambda, \cosh^\rho (s)) \rightarrow  \mathring{C}^{2, \alpha} (\Omega ,  \cosh^\rho (s)) \cap \mathring{C}^k (\Omega \setminus K ,  \ell^{\rho}\cosh^{-1} (s))
\end{align}
such that 
\begin{align} \notag
\Delta_{\Omega} \mathcal R_{\rho \, \alpha}[E] = E.
\end{align}
\end{proposition}

\begin{proof}
Fix $\mathring{E} \in \mathring{C}_0^{0, \alpha} (\Lambda, \cosh^\rho (s))$ and set $\beta : = \|\mathring{E}:  \mathring{C}_0^{0, \alpha} (\Lambda, \cosh^\rho (s)) \| $.  For each integer  $i$, let $A_i $ be the annulus $A_i : = A_0 + i $.  Note that the set $\{ A_i\}$ is a locally finite covering of $\Omega$  such that $A_i \cap A_j = \emptyset$ if $|i - j| >1$. Let $\{ \psi_i\}$ be  a partition of unity subordinate to $\{ A_i \}$ such that $\psi_i (s + 1) = \psi_{i + 1} (s)$. Recall  that $\mathring{E}$ integrates to zero along  meridian circles :
\begin{align*}
\int_{- \pi}^\pi \mathring{E} (s, \theta) d\theta = 0.
\end{align*}
 With  $\mathring{E}_i (s, \theta): = \psi_i\mathring{ E}(s- i, \theta)$, it is straightforward to check that  $\mathring{E}_i \in \mathring{C}_0^{0, \alpha}(A_0)$ with the estimate
\begin{align*}
\| \mathring{E}_i : \mathring{C}_0^{0, \alpha} (A_0)\| \leq  C \beta \cosh^{\rho} (i)
\end{align*}
We then set
\begin{align*}
\mathring{u}_i (s, \theta) : = \mathring{\mathcal{R}}_0 (E_i) (s + i, \theta).
\end{align*}
 From Lemma \ref{LaplaceCPTSPTInverse},
 \begin{align*} 
 \| \mathring{u}_i : C^{2, \alpha} (A_j)\|&  \leq C \beta \cosh^\rho( i)  \cosh^{-1} (j-i) \\ \notag
 &  \leq C \beta
  \left\{\begin{array}{c}
 e^{j} e^{(\rho  - 1)i}, \quad i > j \\ \notag
 \\
 e^{-j} e^{(1 + \rho)i }, \quad i \leq j 
  \end{array} \right.
 \end{align*}
 Finite summation yields 
 \begin{align*}
  \left\| \sum_{i = 0}^n \mathring{u}_i : C^{2, \alpha} (A_j)\right\| \leq \frac{C \beta}{1 - |\rho|} \cosh^\rho (j). 
 \end{align*}
Thus, the partial sums converge to a limiting function $\mathring{u}$ with zero average along meridians satisfying
\begin{align*}
\Delta_{\Omega} \mathring{u} = \mathring{E}, \quad \| \mathring{u}: C^{2, \alpha} (A_j)\| \leq \frac{C \beta}{1 - |\rho|} \cosh^\rho (j). 
\end{align*}
In other words $\mathring{u}$ satisfies the estimate
\begin{align*}
\| \mathring{u}: C^{2, \alpha} (\Omega, \cosh^{\rho} (s))\| \leq \frac{C}{1 - |\rho|} \| \mathring{E}: C^{0, \alpha} (\Omega, \cosh^{\rho} (s))\|
\end{align*}
Setting $\mathcal{R}_{\rho \alpha}[\mathring{E}] : = \mathring{u}$ provides the result.
\end{proof}

We now proceed with the second step in the proof of Proposition \ref{flat_inverse}. Notice the following technical lemma is stated much more generally than we will need. The reader may find it helpful to reduce the statement to when $\rho = 3/4$ and $k =3$.

\begin{lemma} \label{CompactifySupport}
Let $\mathcal{L} : = \Delta_{\Omega} + P (s)$ be a second order linear operator defined on  $\Omega$, and assume that 
\begin{equation} \label{CS0}
\| P (s): C^{0, \alpha} (\Omega, 1 )\| \leq \epsilon. 
\end{equation}
Given $\rho \in (-1 , 1) \setminus \{ 0 \}$, $k \geq 0$, $\alpha \in (0, 1)$ and a compact set $K$ containing $\Lambda$, there exists $\epsilon_*>0$ such that for all $0<\epsilon<\epsilon_*$, there exists a bounded linear map
\begin{align*}
\mathcal{R}_{\rho \, \alpha}[ {\mathcal{L}}, -]: \mathring{C}^{0, \alpha}_0(\Lambda,  \cosh^{\rho} (s)) \rightarrow \mathring{C}^{2, \alpha} (\Omega,  \cosh^\rho (s)) \cap \mathring{C}^k (\Omega \setminus K ,  \ell^{\rho}\cosh^{-1} (s))
\end{align*}
such that
\begin{align} \notag
 \mathcal{L} \mathcal{R}_{\rho\, \alpha}[\mathcal{L}, E]  = E.
 \end{align}
\end{lemma}

\begin{proof}

We proceed by iteration.
Set $u_0 : = \mathcal{R}_{\rho \, \alpha} [E]$ and set
\begin{align*}
\mathcal{L} u_0 = E + P(s) u_0 : = E - E_1
\end{align*}
where $E_1$ is implicitly defined above. From Proposition \ref{LaplaceInverse} $E_1 \in  C^{0,\alpha}(\Lambda)$ and satisfies the estimate
\begin{align*}
\| E_1: \mathring{C}^{0, \alpha}_0 (\Lambda, \cosh^\rho (s)) \|  \leq C(\alpha, \rho) \epsilon\| E: \mathring{C}^{0, \alpha}_0 (\Lambda, \cosh^\rho (s)) \|.
\end{align*}
Taking $\epsilon_*$ sufficiently small we can achieve the estimate, 
\begin{align*}
\| E_1 : \mathring{C}^{0, \alpha}_0 (\Lambda, \cosh^\rho (s)) \| \leq \frac 12 \| E: \mathring{C}^{0, \alpha}_0 (\Lambda, \cosh^\rho (s)) \|.
\end{align*}
Inductively define the sequence $(E_k, u_k)$ by
\begin{align*}
u_k = \mathcal{R}_{\rho \, \alpha }(E_k), \quad E_{k + 1} = E_k - \mathcal{L} u_k .
\end{align*}
Then
\begin{align*}
\| u_i: \mathring{C}^{2, \alpha} (\Omega ,  \cosh^\rho (s)) \cap \mathring{C}^k (\Omega \setminus K ,  \ell^{\rho}\cosh^{-1} (s))\| \leq C \| E_{i}: \mathring{C}^{0, \alpha}_0 (\Lambda, \cosh^\rho (s)) \| \leq C2^{-i}.
\end{align*}
The partial sums $v_n:  = \sum_{i = 0}^n  u_i$ satisfy $\mathcal{L} v_n = E - E_{n + 1}$ and
\begin{align*}
 \| v_n: \mathring{C}^{2, \alpha} (\Omega ,  \cosh^\rho (s)) \cap \mathring{C}^k (\Omega \setminus K ,  \ell^{\rho}\cosh^{-1} (s))\| \leq  C_0  \| E: \mathring{C}^{0, \alpha}_0 (\Lambda, \cosh^\rho (s)) \|
\end{align*} 
where $C_0$ is a uniform constant independent of $n$. The limit $u : = \lim_{n \rightarrow \infty} v_n$ then exists in $C^{2, \alpha}_{loc}$ and by continuity satisfies
\begin{align*}
\mathcal{L} u : = E, \quad \int_{- \pi }^\pi u (s, \theta) d\theta = 0, 
\end{align*}
and the weighted H\"older estimates. The conclusion follows by setting $u : =\mathcal{R}_{\rho \, \alpha} [\mathcal{L}, E]$.
\end{proof}

As the operator $\cosh^2(s) \mathcal L_F$ does not satisfy the condition \eqref{CS0}, to use a perturbation technique we must modify this operator by cutting off the curvature term for $s$ near zero. 
\begin{definition}\label{ModifiedOperator}
Given $a \in \mathbb{R}$, we define the operator $\hat{\mathcal{L}}_{F\, a}$ as follows:
\[
\hat{\mathcal{L}}_{F \,a} : = \Delta_{\Omega} + 2 \psi[a - 1, a  ] (|s|) \cosh^{ - 2} (s).
\]
\end{definition}

From the definition and the properties of $\mathcal L_F$, it immediately follows that

\begin{lemma} 
\begin{enumerate}
\item $\hat{\mathcal{L}}_{F \, a} = \cosh^{2} (s)\mathcal{L}_F$ on $\Omega \cap\{|s| \geq a \}$.
\item For any $\epsilon> 0$ there exists an $a$ sufficiently large so that
\begin{align*}
\|\hat{\mathcal{L}}_{F \, a} - \Delta_\Omega, C^{0, \alpha} (\Omega, \cosh^{-1} (s))   \| \leq \epsilon.
\end{align*}
\item The operator $\hat{\mathcal{L}}_{F \, a}$ preserves the class of functions with zero average over meridian circles on $\Omega$.
In other words, given $f \in \mathring{C}^{k, \alpha} (\Omega)$, $\hat{\mathcal{L}}_{F \, a} f \in \mathring{C}^{k - 2, \alpha} (\Omega)$. 
\end{enumerate}
\end{lemma}

We are now ready to quantify the error induced by inverting $\hat{\mathcal L}_{F\, a}$ rather than $\cosh^2(s)\mathcal L_F$. We describe the error via the following map. Note throughout what follows that in all applications we will need, $\rho=3/4, \alpha=3/4$.

\begin{definition} \label{truncation_map}
 The map $\mathfrak{T}[-]:  \mathring{\mathcal{X}}^{k}_0 \rightarrow C^{k, 3/4} (\Omega, \cosh^{ 3/4}(s)) $ is given as follows:
\begin{align}\notag
\mathfrak{T}[E] (s, \theta) : = E -  \cosh^2 (s) \mathcal{L}_F  (s)\mathcal{R}_{\rho \, \alpha} [\hat{\mathcal{L}}_{F\, a} ,  E] (s, \theta).
\end{align}
\end{definition}

\begin{proposition}\label{TruncationMapProperties}
The following statements hold:
\begin{itemize}
\item[(1)] For $a$ sufficiently large, $\mathfrak{T}[-]$ is well-defined.
\item[(2)] $\mathfrak{T}[E]$ is compactly supported on the set $\Lambda \cap\{|s| \leq a \}$.
\item[(3)] There is a constant $C(\rho)$ depending only on $\rho$ so that 
\begin{align} \notag
\| \mathfrak{T}[E] : C^{0,   3/4} (\Lambda \cap\{|s| \leq a \}, 1 )\| \leq C(\rho) \| E: \mathcal{X}^{0} \|.
\end{align}
\item[(4)] $\mathfrak{T}$ maps $\mathring{\mathcal{X}}^{0 \, \perp}_0$ into $\mathring{\mathcal{X}}^{0 \, \perp}_0$ where $E \in \mathring{\mathcal{X}}^{0 \, \perp}_0$ if $E \in {\mathcal{X}}^{0 \, \perp}_0$ and $E$ has zero radial average along meridian circles.
\end{itemize}
\end{proposition}
\begin{proof}
Statements $(1)$ and $(2)$ are obvious from Definition \ref{truncation_map}. Statement $(3)$ follows directly from  Lemma \ref{CompactifySupport} and the  definition  of the maps $\mathcal{R}_{\rho, \alpha}[\mathcal{L}, -]$. To see $(4)$, note first that the zero average condition on meridians is preserved by  the definition of $\mathfrak{T}[-]$, Lemma \ref{CompactifySupport} and the definition of $\hat{\mathcal{L}}_{F \, a}$. We then immediately get that
\begin{align}\notag
\int_\Lambda \mathfrak{T} [E] \kappa_z d \mu_\Omega= 0,
\end{align} 
since $\kappa_z  = \kappa_z (s)$ is $\theta$-independent.  Additionally, we have
\begin{align}\notag
\int_\Lambda \mathfrak{T}[E] \kappa_x d\mu_{\Omega} & = \int_{\Lambda} E\kappa_x d \mu_{\Omega}  +  \int_\Lambda \mathcal{L}_F \hat{\mathcal{R}}_{\rho \, a} [E]  \kappa_x d\mu_{F} \\ \notag
\end{align}
Note that by assumption the first term above is zero when $E$ is in $\mathcal{X}^{k \, \perp}$. Considering the support of $E$ and the definition of $\Lambda$, for $L \geq \ell $, $\Lambda \subset \Omega_L$. (Recall the definition of $\Omega_L$ in (\ref{OmegaLDef}).) Therefore,
\begin{align}\notag
  \int_\Lambda \mathcal{L}_F \hat{\mathcal{R}}_{\rho \, a} [E]  \kappa_x d\mu_{F} & = \lim_{L \rightarrow \infty} \int_{\Omega_L} \mathcal{L}_F\mathcal{R}_{\rho \, \alpha} [\hat{\mathcal{L}}_{F\, a} ,  E]  \kappa_x d\mu_{F}  \\ \notag
  & =  \lim_{L \rightarrow \infty} \int_{- \pi}^\pi \left(\kappa_x (L, \theta) \nabla_s  \mathcal{R}_{\rho \, \alpha} [\hat{\mathcal{L}}_{F\, a} ,  E]  (L, \theta)- \mathcal{R}_{\rho \, \alpha} [\hat{\mathcal{L}}_{F\, a} ,  E]  (L, \theta)\nabla_s \kappa_x  (L, \theta)\right) d \theta,
\end{align}
where we have used that $\mathcal L_F\kappa_x=0$.  Considering the growth rates of $\mathcal{R}_{\rho \, \alpha} [\hat{\mathcal{L}}_{F\, a} ,  E]$ and $\kappa_x$, we see the right hand side above is equal to zero. The claim then follows immediately. 
\end{proof}

We are now ready to prove Proposition \ref{flat_inverse}.
\begin{proof}[Proof of Proposition \ref{flat_inverse}]
Let $E$ be a function in $\mathcal{X}^{0 \, \perp}_0$ and set $\beta : = \| E : \mathcal{X}^{0 \, \perp}\|$. Set
\begin{align}\notag
\bar{E} (s) : = \frac{1}{2\pi}\int_{- \pi}^\pi E(s, \theta) d\theta, \quad \mathring{E} (s, \theta) : = E(s, \theta) - \bar{E} (s).
\end{align}
It is then straightforward to verify that  $\mathring{E}$ is in $\mathring{\mathcal{X}}^{0\, \perp}_0$ and
\begin{align}\notag
\| \mathring{E} : \mathring{\mathcal{X}}^{0\, \perp}_0 \| \leq C \beta.
\end{align}
The equation 
\begin{align}\notag
\cosh^{2}(s) \mathcal{L}_F u = \mathfrak{T}[\mathring{E}]
\end{align}
on $\Lambda$ is equivalent to 
\begin{align} \label{SphereProblem}
\left(\Delta_{\mathbb{S}^2} + 2\right) u = 2\mathfrak{T}[\mathring{E}]/(|A_F|^2 \cosh^{2} (s))
\end{align}
on the sphere, where we have abused notation slightly and identified functions with their lifts to the sphere under the Gauss map of $F$. Since $\mathfrak{T}[\mathring{E}]$ is supported on the set $\Omega \cap \{ |s| \leq a\}$, 
\begin{align}\notag
\|\mathfrak{T}[\mathring{E}]/(|A_F|^2 \cosh^2(s)): L^2 (\mathbb{S}^2) \| \leq C \|\mathfrak{T}[\mathring{E}]/|A_F|^2: \mathring{\mathcal{X}}^{0 \, \perp} \| \leq C \beta.
\end{align}
Moreover, by Proposition \ref{TruncationMapProperties} $(4)$, $\mathfrak{T}[\mathring{E}]/(|A_F|^2 \cosh^{2}(s))$ is orthogonal to the kernel of $\Delta_{\mathbb{S}^2} + 2$ (Recall that the kernel is spanned by the coordinate functions $\bar{x}$, $\bar{y}$  and $\bar{z}$). Thus, there is a solution $u$ to (\ref{SphereProblem}) satisfying
\begin{align}\notag
\| u :  W^{2, 2} (\mathbb{S}^2)\| \leq C\beta.
\end{align}
Standard elliptic theory then implies
\begin{align}\notag
\| u \|_{2, \alpha} \leq C \|  \mathring E\|_{0, \alpha}. 
\end{align}
Put 
\begin{align}\notag
\mathring{v} : = u +  \mathcal{R}_{\rho \, \alpha} [\hat{\mathcal{L}}_{F\, a} ,  \mathring E].
\end{align}
Then $\mathring{v}$ satisfies $\cosh^{2}(s)\mathcal{L}_F \mathring{v} = \mathring E$ and the estimate
\begin{align}\notag
\|\mathring{v} : \mathring{\mathcal{X}}^{2 }\| \leq C \| \mathring E: \mathring{\mathcal{X}}^{0\, \perp} \|.
\end{align}
The function $\bar{E}$ can then be solved for by direct integration. In particular, in \cite{BK} we proved that the the expression
\begin{align} \label{MIT1}
\bar{v} (s) : = \left(\int_0^s \tanh^{-2} (s) \int_0^{s'} \tanh(s'') \cosh^{-2}(s) \bar{E}(s'') d s'' ds' \right) \tanh(s)
\end{align}
satisfies $\cosh^2(s)\mathcal{L}_F \bar{v} = \bar{E}$. Moreover, by Lemma 8 of \cite{BK}, making the appropriate modifications for the norms of interest here, we establish there exists a uniform $C>0$ such that
\[
\| \bar v:\mathcal X^2\| \leq C \|\bar E: \mathcal X^0\|.
\]
We conclude by setting
\begin{align}\notag
v : = \bar{v} + \mathring{v} \text{ and } \mathcal{R}_F^\perp [E] : = v.
\end{align}

\end{proof}

\subsection{Solving the linear problem on $\mathcal X_0^0$}
We now solve the more general linear problem by decomposing any $E \in \mathcal X_0^0$ into a $\theta$ independent function and a function $\mathring E$, with zero average along meridian circles. We invert the $\theta$ independent function directly using \eqref{MIT1}. To invert $\mathring E$, we modify the function by its projection onto the space $\underline \kappa$. The resulting function is orthogonal to $\underline \kappa$ and thus can be inverted via Proposition \ref{flat_inverse}.

\begin{proposition}  \label{FlatTotalInverse}
 There is a bounded  linear map 
\begin{align}\notag
\mathcal{R}_F[-] : \mathcal{X}^{0}_0 \rightarrow \mathcal{X}^{2}  \times\mathbb {R}^2
\end{align}
such that for $E \in \mathcal{X}^0_0$ and $(v, b_x, b_y) : = \mathcal{R}_F [E]$, 
\begin{align}\notag
\cosh^{2}(s) \mathcal{L}_F v= E - b_x w_x - b_y w_y.
\end{align}

\end{proposition}
\begin{proof}
Given $E \in \mathcal{X}^0_0 $, we set 
\begin{align}\notag
\bar{E}(s) : =  \int_{-\pi}^\pi E(s, \theta) d \theta, \quad \mathring{E} (s, \theta) = E(s, \theta) - \bar{E}(s). 
\end{align}
Then $\mathring{E}$ is in $\mathring{\mathcal{X}}^0_0$. Since $\kappa_z$ is $\theta$-independent and $\mathring E$ has zero average along meridians,
\[
\int_{\Lambda} \mathring E \, \kappa_z \, d \mu_\Omega = 0.
\]Moreover, the definition of the function space $\mathcal X^0$ and the definition of $\Lambda$ together imply
\[
\left|\int_{\Lambda} \mathring E \,  \kappa_x \,  d \mu_\Omega \right|  + \left|\int_{\Lambda} \mathring E \,  \kappa_y \,  d \mu_\Omega \right| \leq C \| \mathring E: \mathring{\mathcal{X}}^{0} \|.
\]

Thus, there exist constants $b_x$ and $b_y$ with 
\begin{align}\notag
|b_x|, \quad |b_y| \leq C \|E : \mathcal{X}^{0} \|
\end{align}
so that $ E^\perp : = \mathring E - b_x w_x - b_y w_y$ lies in $\mathcal{X}^{0 \, \perp}_0$. We then set
\begin{align}\notag
v^\perp : = \mathcal{R}^\perp_F [E^\perp]
\end{align}
and we define $\bar{v}(s)$ by expression (\ref{MIT1}). The function $v : = v^\perp + \bar{v}$ then solves
\begin{align}\notag
\cosh^{2}(s) \mathcal{L}_F v = E - b_x w_x - b_y w_y 
\end{align}
and
\[\|v:\mathcal X^2\|+ |b_x| + |b_y| \leq C\|E:\mathcal X_0^0\|.
\]
\end{proof}

\section{Finding exact solutions}   \label{FindingExactSolutions}

Recall that finding a graph over $G$ so that the resulting surface has $H=0$ is equivalent to finding a function $u \in C^{2,\alpha}$ with $Q[u]=0$. We will solve this problem via standard gluing methods, invoking a fixed point theorem for a given map $\Psi$ from some Banach space we designate.  Estimates for the map will require a strong understanding of $Q[u]$ and to that end we first consider a natural decomposition of $Q$.
\begin{definition} \label{QLinQuadDefs}
Let $\mathcal{L}_Q[u]$ denote the linearization of the operator $Q$ at $0$, and set
\begin{align}\notag
R_Q [u]  : = Q[u] - Q[0]-\mathcal{L}_Q[u]. 
\end{align}
\end{definition}
We first record linear estimates, which are controlled by properties of the immersions $F$ and $G$.
\begin{lemma}
Given an  anti-symmetric matrix $\RRR$ and $\xi \in \mathbb{R}$, choose $\delta >0$ such that $C(1+|\RRR| + \xi) \delta < \tilde \epsilon$ where $\tilde \epsilon$ is from Proposition \ref{HQEstimates} and $C$ is a universal constant arising from the norm bounds in the linear problem.
Then 
\begin{align} \label{XSP0}
\|Q [0]: \mathcal{X}^{0} \|  \leq C \delta \ell^{1/4}|\RRR|
\end{align}and for $u \in \mathcal X^2$,
\begin{equation}\label{linearerroreq}
\| \mathcal{L}_Q[u] - \cosh^{2}(s) \mathcal{L}_Fu :\mathcal X^0\| \leq C\delta(1 + |\RRR|+|\xi|)  \|u:\mathcal X^2\|.
\end{equation}
\end{lemma}
\begin{proof}Proposition \ref{normalized_problem_properties} implies
\begin{align*}
\|Q[0] \|_{j, 3/4} (s, \theta) < C\delta |\RRR| \cosh(s)
\end{align*}
and thus
\begin{align*}
\cosh^{-3/4} (s)\| Q[0]\|_{j, 3/4} (s, \theta)\leq C \delta |\RRR|\cosh^{1 - 3/4} (s) \leq  C \delta \ell^{1/4} |\RRR|.
\end{align*}For the second estimate,
recall that 
\begin{align*}
Q[u] = \cosh^{2}(s)H(\tilde{\underline{\nabla}} [G] + \mathcal{E}_{\delta} [u_{0} + u])
\end{align*}
Then for $\underline{\nabla}_0 : = \tilde{\underline{\nabla}}[G] + \mathcal{E}_\delta [u_0]$,
\begin{align*} 
&\mathcal{L}_Q [u] = \dot{Q}[u] =  \cosh^2(s) \left. D H \right|_{\underline{\nabla}_0} (\mathcal{E}_{\delta}[u]) \\
&\mathcal{L}_F u = \left. D H \right|_{\underline{\nabla}F} (\underline{\nabla} [u \,  \nu_F]).
\end{align*}
We then write
\begin{align*}
\mathcal{L}_Q [u] - \cosh^{2} (s) \mathcal{L}_F u & = \cosh^{2} (s) \left\{ \left. D H\right|_{\underline{\nabla}_0}(\mathcal{E}_{\delta}[u]) - \left. D H \right|_{\underline{\nabla} F} (\mathcal{E}_{\delta}[u])\right\} 
- \cosh^{2} (s) \left. D H \right|_{\underline{\nabla} F} \left(\underline{\nabla}[ u\,\nu_F] - \mathcal{E}_{\delta}[u] \right) \\
& : = I + II.
\end{align*}
Lemma \ref{close_to_a_straight_line},   \eqref{G_0Def} and \eqref{G_0Normalized}, and the triangle inequality imply
\begin{align} \label{SSP3}
\| \underline{\tilde{\nabla}} [G] - \underline{\nabla} F : C^{k} (\Lambda, \cosh(s))  \|
\leq C \delta + C\delta |\RRR|.
\end{align}
Moreover, the estimate for $u_{0 }$ in Theorem \ref{BK1Theorem} (with $\delta \xi$ replacing $\delta$) and Definition \ref{FunctionCommutatorDef} imply
\begin{align} \label{SSP5}
\| \mathcal{E}_{\delta} [u_{0 }] : C^{0, \alpha}(\Lambda,\cosh(s) )\| \leq C \delta|\xi|.
\end{align}
Combining \eqref{SSP3} and \eqref{SSP5}, we note that
\begin{align} \notag
 \| \underline{\nabla}_0 - \underline{\nabla} F : C^{0, \alpha} (\Lambda, \cosh(s))\|
 \leq C \delta(1 + |\RRR|+|\xi|)  < \tilde \epsilon. 
\end{align}
Thus we may apply Proposition \ref{HQEstimates}  with $\underline{\nabla}  = \underline{\nabla} F$ and $\mathcal{E}  =  \underline{\nabla}_0 - \underline{\nabla} F$. Notice here that $\mathfrak{ a}(\underline \nabla) =1$, $|\nabla| = \cosh(s)$, and $d=-2$. Thus
\begin{align}\notag
\|  I\|_{0, \alpha} \leq C \cosh^{-1}(s)\| \mathcal E\|_{0, \alpha} \|  u \|_{2, \alpha} \leq C\delta  \cosh^{-1}(s)(1 + |\RRR|+|\xi|)  \| u\|_{2, \alpha}.
\end{align}
Again using Definition \ref{FunctionCommutatorDef} and estimates on the derivatives of $\nu_G$, we observe that 
\begin{align*}
\| \underline{\nabla}[ u\,\nu_F] - \mathcal{E}_{\delta}[u] \|_{0, \alpha} \leq C\delta(1+|\RRR|)  \| u \|_{2, \alpha}
\end{align*}
and it follows that $\| II \|_{0, \alpha} \leq C\delta(1+|\RRR|)  \| u\|_{2, \alpha}$.
\end{proof}
We now define the Banach space on which we will solve the fixed point theorem.

\begin{definition} \label{SchauderSet}
For $\zeta \gg 1$, set
\begin{align} \notag
\Xi : = \{ (u, b_x, b_y ) \in \mathcal{X}^{2} \times \mathbb R^2: \|u : \mathcal{X}^2\| \leq \zeta \delta \ell^{1/4} |\RRR|, |b_x|, |b_y| \leq \zeta \delta \ell^{1/4}|\RRR| \}.
\end{align}
\end{definition}
\noindent Notice that by definition, $\Xi$ is a compact, convex subset of $\mathcal{X}^2 \times \Real^2$. Thus, we are in a setting where it is natural to apply Schauder's fixed point theorem.

Now we control the higher order terms of $Q[u]$ where $u = v + b_xu_x + b_yu_y$ with $(v,b_x,b_y) \in \Xi$.
\begin{proposition} \label{SchauderSetProps}
Given $\zeta \gg 1$, $\xi \in \Real$, and a non-trivial anti-symmetric matrix $\RRR$, choose any $0<\delta|\xi|<\delta_0$ and $\ell>16$ such that $C \delta(1+|\RRR|+|\xi|)\ell^{5/4} <\min\{1/(4\zeta),\tilde \epsilon /\zeta\}$. Here $C$ is a universal constant arising from the norm bounds in the linear problem and $\delta_0$ comes from Theorem \ref{BK1Theorem}. Given $(v, b_x, b_y) \in \Xi$, and $f : = v + b_x u_x + b_y u_y $
\begin{align} \notag
\| R_Q[ f ]: \mathcal{X}^{0} \| \leq C \zeta^2 \delta^2  |\RRR|^2  \ell^{3/4} .
\end{align}
\end{proposition}

\begin{proof}

The definition of $R_Q[f]$ implies
\begin{align} \notag
R_Q[f] = \cosh^2 (s) R_{H, \mathcal{E}_\delta [f]}^{(1)} (\underline{\nabla}_0).
\end{align}
From Definition \ref{FunctionCommutatorDef},
\begin{align} \notag
\|\mathcal{E}_{\delta}[f] \|_{0, \alpha}&  \leq C \| f \|_{2, \alpha} \leq C\| v\|_{2, \alpha} + \left(|b_x| + |b_y| \right) \cosh(s) \\ \notag
&\leq C\zeta  \delta |\RRR| \ell^{1/4}\left( \cosh^{3/4} (s) + \cosh(s) \right)\\&\leq C\zeta  \delta |\RRR| \ell^{1/4}\cosh(s)\notag.
\end{align}
Applying Proposition \ref{HQEstimates} with $\underline{\nabla} = \underline{\nabla}_0$ and $\mathcal{E} = \mathcal{E}_\delta[f]$,  $\Phi=H$ and $d=-1$ implies
\begin{align*}
\cosh(s) \| R_Q[f] \|_{0, \alpha} \leq C  \zeta^2  \delta^2 |\RRR|^2 \ell^{1/2} \cosh^2(s).
\end{align*}
As $|s| \leq \arccosh (\ell)$ on $\Lambda$, the estimate immediately follows.
\end{proof}

Using all of the previous estimates, we define a map $\Psi$ that takes $\Xi$ into $\Xi$. As previously mentioned, because we chose to simplify the linear problem by considering only inhomogeneous terms with zero boundary data on $\Lambda$, we now have a small amount of technical work to do. In particular, to invert $Q[u]$, we need to first cut it off by a function $\psi'$. 
\begin{proposition} \label{SchauderMapProperties}
For $\psi(s):=\psi [\arccosh (\ell/2),\arccosh ( \ell/4)](|s|)$, and $\psi'(s):= \psi[\arccosh (\ell), \arccosh (\ell/2)](|s|)$,
set
\begin{align}\notag
\Psi[v, b_1, b_2] = (\psi v, b_1, b_2) - \mathcal{R}_F[\psi'Q[ \psi v + b_1 u_x + b_2 u _y]]
\end{align}
Given $\zeta \gg 1$, $\xi \in \Real$, and a non-trivial anti-symmetric matrix $\RRR$, for any $0<\delta|\xi|<\delta_0$ and $\ell>16$ such that $C \delta(1+|\RRR|+|\xi|)\ell^{1/2} <\min\{1/(4\zeta),\tilde \epsilon /\zeta\}$, $\Psi(\Xi) \subset \Xi$.

\end{proposition}
\begin{proof}

We begin by noting
\begin{align}\notag
 Q[\psi v + b_1u_x + b_1 u_y] & = Q[0] +\mathcal{L}_Q [  \psi v + b_1u_x + b_2 u_y] +  R_Q[ \psi v + b_1u_x + b_2 u_y] \\ \notag
& =  Q[0] + \cosh^{2}(s) \mathcal{L}_F (\psi  v + b_1u_x + b_2 u_y) +  R_Q[ \psi v + b_1u_x + b_2 u_y] \\ \notag
& \quad + \left( \mathcal{L}_Q - \cosh^2(s)\mathcal{L}_F\right) (\psi v + b_1 u_x + b_2 u_y).\\ \notag
\end{align}

First define $(u^0, b_{x}^0, b_{ y}^0) : = \mathcal{R}_F [ \psi'Q[0]]$.
 Proposition \ref{FlatTotalInverse}, the estimate \eqref{XSP0}, and the uniform bounds on $\psi'$ imply
\begin{align} \label{SMP1}
\|u^0: \mathcal{X}^2\|+ |b_{x}^0|+ |b_{y}^0| \leq C\|\psi' Q[0]: \mathcal X^0_0\| \leq C \delta |\RRR| \ell^{1/4}.
\end{align}
Set $(u', b_x', b'_y) :  = \mathcal{R}_F[\psi'   R_Q[ \psi v + b_1u_x + b_2 u_y]] $. Then, Propositions \ref{FlatTotalInverse} and \ref{SchauderSetProps} imply
\begin{align} \label{SMP3}
\|u': \mathcal{X}^2\|+ |b'_x|+ |b'_y|  \leq C\|  \psi'  R_Q[ \psi v + b_1u_x + b_2 u_y] : \mathcal{X}^0_0\| \leq C \zeta^2 \delta^2 |\RRR|^2 \ell^{3/4}.
\end{align}
 Finally, set
 \begin{align*}
 R'' : = \left( \mathcal{L}_Q - \cosh^2(s)\mathcal{L}_F\right) (\psi v + b_1 u_x + b_2 u_y)
 \end{align*}and $(u'', b_x'', b_y'') : = \mathcal{R}_F [\psi'R'']$. By Proposition \ref{FlatTotalInverse}, \eqref{linearerroreq}, the estimates for $v, b_1, b_2$, and the decay control on $\psi$,
 \begin{align} \label{SMP4}
\|u'': \mathcal{X}^2\|+ |b''_x|+ |b''_y| \leq C\| R'' : \mathcal{X}^0_0 \| \leq C  \delta^2(1+|\RRR|+|\xi|)\zeta |\RRR| \ell^{ 1/2}\leq C \delta \zeta |\RRR| \ell^{ 1/2}.
 \end{align} The definitions of $u_x, u_y$ imply $\mathcal L_F (b_1u_x+b_2u_y)=0$ for all $|s| \geq 2$. Moreover, as $\psi v \equiv 0$ for all $s$ with $\cosh(s) > \ell/2$, $\psi'\mathcal L_F (\psi v+ b_1u_x+b_2u_y)=\mathcal L_F (\psi v+ b_1u_x+b_2u_y)$. In other words,
 \[
 \mathcal R_F[\psi'\mathcal L_F (\psi v+ b_1u_x+b_2u_y)] = (\psi v, b_1, b_2).
 \]
Therefore
\begin{align*} 
\Psi(v, b_1, b_2) = (- u^0 - u' - u'', - b^0_{ x} - b_x' - b_x'', - b^0_{ y} - b_y' - b_y'').
\end{align*}
Combining (\ref{SMP1}),  (\ref{SMP3}) and (\ref{SMP4}), as long as $\zeta>4C$, the hypothesis on $\delta$ allows us to conclude
\begin{align}
\| - u_0 - u' - u'': \mathcal{X}^2 \|\leq C \delta |\RRR| \ell^{1/4} + C \zeta^2 \delta^2 |\RRR|^2 \ell^{3/4} + C \zeta \delta |\RRR| \ell^{ 1/2} \leq  3\zeta \delta \ell^{1/4}|\RRR|/ 4 
\end{align}Similar estimates hold for the coefficients of $u_x, u_y$:
\begin{align}
| b^0_{x} - b_x' - b_x''| \leq  3\zeta \delta \ell^{1/4}|\RRR|/ 4 , \quad  | b^0_{ y} - b_y' - b_y''| \leq  3\zeta \delta \ell^{1/4}|\RRR|/ 4. 
\end{align}
\end{proof}

\section{The Main Theorem}\label{MainTheorem}
We prove the main theorem via a fixed point argument and all of the necessary estimates have now been recorded. We require one further technical proposition which allows us to prove embeddedness in the main theorem by considering graphs over $G$ as coming from graphs over $F$.
\begin{proposition} \label{SubstituteKernelGraphsAreEmbedded}
There is a  constant $\epsilon_F > 0$ so that: Given constants $b_x$ and $b_y$ and a vector field $\mathbf v: \mathbb R^2 \rightarrow \mathbb{R}^{3}$, satisfying 
\begin{align}\notag
|b_x|, \, |b_y| \leq \epsilon_F, \quad \|\mathbf v :C^1(\mathbb R^2)\| \leq \epsilon_F, 
\end{align}
 the surface $F + \mathbf v +  (b_x u_x + b_y u_y)\nu_F$ is an  embedded surface in $\mathbb{R}^3$.
\end{proposition}
\begin{proof}
In the following, we refer to the angle that the projection of a vector onto the plane $\{ z = 0\}$ makes with the vector $\e_x$ as the argument, and for a vector $\mathbf v$ we denote it by $Arg (\mathbf v)$. Recall that 
\begin{align} \label{SKGAE1}
F(s, \theta) = \sinh(s) \e_r (\theta) + \theta \e_z \\ \notag
\nu_F(s, \theta) = - \cosh^{-1} (s) \e_r^\perp (\theta) + \tanh(s) \e_z
\end{align}
where we have denoted $\e_r(\theta) : = \sin (\theta) \e_x + \cos (\theta) \e_y$ and $\e_r^\perp(\theta)  = \cos (\theta) \e_x - \sin (\theta) \e_y$. Set $\tilde{F} : = F +\mathbf v + (b_x u_x + b_y u_y) \nu_F$. Then it follows from (\ref{SKGAE1}) that 
\begin{align} \label{SKGAE2}
|Arg(\tilde{F}(s, \theta)) - Arg(F(s, \theta)) | \leq C \epsilon.
\end{align}
Let $x_i = (s_i, \theta_i)$, $i = 1, 2$ be two points such that $\tilde{F} (x_1) = \tilde{F}(x_2)$ and assume first that $s_1 > 0$, $s_2 < 0$. Then it follows from (\ref{SKGAE2}) that $|\theta_1 - \theta_2 - (2 k  + 1)\pi| \leq C\epsilon.$ We can write 
\begin{align} \notag 
u_x \nu_F(s, \theta) =   -\frac{1}{4 \pi}\psi_0 (s)\cos (\theta)  \e_r^\perp (\theta)+ \frac{1}{4 \pi}\psi_0 (s)\cos (\theta)  \sinh(s) \e_z, \\ \notag
u_y \nu_F(s, \theta) =  - \frac{1}{4 \pi}\psi_0 (s) \sin (\theta) \e_r^\perp (\theta) +  \frac{1}{4 \pi}\psi_0 (s) \sin (\theta) \sinh(s) \e_z.
\end{align}
This then gives that 
\begin{align} \label{SKGAE4}
\left|\left.\left( b_xu_x  + b_y u_y \right) \nu_F\right|_{x_1}^{x_2} \right|^2 \leq C \epsilon^2 \left(1 +  (\sinh(s_1) + \sinh(s_2))^2 \right)
\end{align}
Moreover, from \eqref{SKGAE1}, we have 
\begin{align}\label{SKGAE5}
\left| F(x_2) - F(x_1) \right|^2 \geq (1 -  C \epsilon) (\sinh(s_2) + \sinh(s_1))^2 +  ( \pi - 2 \epsilon)^2
\end{align}
Combining \eqref{SKGAE4} and \eqref{SKGAE5} implies
\begin{align}\notag
|\tilde{F} (x_2)-\mathbf v(x_2) - \tilde{F}(x_1)+\mathbf v(x_1)| \geq \pi - C \epsilon.
\end{align}Here $C$ is just a universal constant depending on $F$. Now suppose $|s_1-s_2|  \leq 10$. Then the uniform $C^1$ bounds on $\mathbf v$ imply
\[
|\tilde F(x_2) - \tilde F(x_1)| \geq |\tilde{F} (x_2)-\mathbf v(x_2) - \tilde{F}(x_1)+\mathbf v(x_1)|- |\mathbf v(x_2) -\mathbf v(x_1)| \geq \pi - C \epsilon.
\]On the other hand, if $|s_1-s_2|>10$, then the $C^0$ estimate on $\mathbf v$ is enough as in this case,
\[
|\tilde F(x_2) - \tilde F(x_1)| \geq |\tilde{F} (x_2)-\mathbf v(x_2) - \tilde{F}(x_1)+\mathbf v(x_1)|- |\mathbf v(x_2) -\mathbf v(x_1)| \geq \pi - C \epsilon - |\mathbf v(x_1)| -|\mathbf v(x_2)|\geq \pi - C \epsilon.
\]Since $C$ depends only on $F$, we can choose $0<\epsilon<\epsilon_F$ so that $\pi - C\epsilon>0$ and thus no self-intersections exist. 
A similar argument gives the same result when both $s_1$ and $s_2$ are positive. The result then follows immediately. 
 \end{proof}

We now gather all of the estimates from the previous section in combination with the previous proposition to prove the main theorem.
\begin{theorem}
Given an anti-symmetric matrix $\RRR$, $ \xi \in \Real$  and $\zeta \gg 1$, choose any $0<\delta<\delta_0/|\xi|$ and $\ell > 16$ such that $C \delta(1+|\RRR|+ |\xi|)\ell<\min\{1/(4\zeta),\tilde \epsilon /\zeta, \epsilon_F/(4\zeta)\}$. Here $C$ is a universal constant arising from the norm bounds in the linear problem and $\delta_0$ comes from Theorem \ref{BK1Theorem}. There exists $(v,b_x, b_y) \in \Xi$ such that
the surface $G+ e^{\delta \xi \theta}\left(v + b_x u_x + b_y u_y + u_{0}\right)\nu_G$ 
\begin{enumerate}
\item is an immersed smooth surface with boundary. 
\item is minimal on the domain $\{(s,\theta) | \cosh(s) \leq \ell/4\}$.
\item is embedded when  
\begin{align} \notag
\ell \leq \frac{1}{\delta \xi} \left( \frac{e^{\pi \xi/ \rho_0} - 1}{e^{\pi \xi/\rho_0} + 1}\right) \sqrt{ \frac{\tau_0^2 + \xi^2}{\rho_0^2 + \xi^2}}.
\end{align}
\end{enumerate}

\end{theorem}
To conclude the theorem found in the introduction, simply set $\epsilon_1 \leq \min\{1/(4C\zeta),\tilde \epsilon /(C\zeta), \epsilon_F/(4C\zeta)\}$.
\begin{proof}

To prove embeddedness, note that for $|z| \leq \frac 1\delta$, extending the calculation of \eqref{CC2} to consider also $D^2M$ implies
\begin{align*}
\|M (x, y, z) -(x,y,z):C^1(\{|z| \leq  1/\delta\})\|=  O (\delta (z^2 + xz + yz)(1+|\RRR|)).
\end{align*}
Thus for $|\theta| \leq 4\pi$ and $\Lambda_{4\pi}:= \{(s,\theta) | (s,\theta) \in \Lambda, |\theta| \leq 4\pi\}$, 
\begin{align*}
\|G(s, \theta) - F(s, \theta):C^1(\Lambda_{4\pi}) \| \leq \delta O (1 + \cosh(s)(1+ |\RRR|)).
\end{align*}
Let $\mathbf v:=  e^{ \delta \xi \theta} f \nu_G$ and $\mathbf w = e^{ \delta \xi \theta} f\left(\nu_G-\nu_F\right)$. Then\begin{align*}
\|G + e^{ \delta \xi \theta} f \nu_G - \left( F +  e^{ \delta \xi \theta} f \nu_F\right):C^1(\Lambda_{4\pi})\| &= 
\|G + \mathbf v-\left( F-\mathbf v\right)+ \mathbf w:C^1(\Lambda_{4\pi})\| \\&\leq \delta O(1 +\cosh(s)(1+|\RRR|)) +\|\mathbf w:C^1(\Lambda_{4\pi})\|.
\end{align*}
Taking $f$ of the form $f = v + u_x b_x + u_y b_y+u_0$ for $(v, b_x, b_y) \in \Xi$, and using the estimate of \eqref{SSP3}, 
\begin{align*}
\|\mathbf w:C^1(\Lambda_{4\pi})\|& \leq \|f :C^1(\Lambda_{4\pi})\| \|\nu_G - \nu_F:C^1(\Lambda_{4\pi})\|\\
&\leq C\delta^2 \zeta(|\RRR|+ |\RRR|^2) \ell^{5/4}.
\end{align*}Thus, the graph over $G$ can be viewed as a graph over $F$ by $(b_xu_x+b_yu_y)\nu_F + \tilde{\mathbf  w}$ where $\tilde{\mathbf  w} = (v+u_0)\nu_F + \mathbf w+(G-F)$. The definition of $\Xi$ implies
\[
|b_x| + |b_y| \leq \epsilon_F/2.
\]As $\delta \ell <1$, the perturbation vector field has the bound 
\begin{align*}
\|\tilde{\mathbf w}:C^1(\Lambda_{4\pi})\| & \leq \|G-F+\mathbf w:C^1(\Lambda_{4\pi})\| + \|(v + u_0)\nu_F:C^1(\Lambda_{4\pi})\|\\
&\leq C\left(\delta \ell(1+|\RRR|) + \delta \zeta (1+|\RRR|)\ell+\delta \zeta |\RRR| \ell^{1/4}\right)\\
& \leq\epsilon_F.
\end{align*}
The self-similarity of the curve implies that, up to some translation, the same argument can be done for any domain of $\Lambda$ with $\theta \in [\tilde \theta - 2\pi, \tilde \theta + 2\pi]$. Applying Proposition \ref{SubstituteKernelGraphsAreEmbedded}, we conclude the surface is locally an embedding. We appeal to Proposition \ref{CurveClassification} to get global embeddedness.

The minimality will follow from a fixed point argument. Namely, since $\Phi(\Xi) \subset \Xi$ and $\Xi$ is a compact, convex subset of a Banach space, there exists $(v,b_x, b_y) \in \Xi$ such that
\[
(v,b_x, b_y) = (\psi v, b_x, b_y) - \mathcal R_F[\psi'Q[\psi v + b_xu_x+b_yu_y]].
\]This implies $\mathcal R_F[\psi'Q[\psi v + b_xu_x+b_yu_y]]=((1-\psi)v,0,0)$ and thus on the region where $\psi = 1$,
\[
Q[ v + b_xu_x+b_yu_y]=0.
\]The definition of $Q$, see \eqref{Qdef}, implies
\[
H(\underline{\tilde \nabla }G + \mathcal E_\delta[v + b_xu_x+b_yu_y+u_0])=0
\] and thus 
$G+ e^{\delta \xi \theta}\left(v + b_x u_x + b_y u_y + u_{0}\right)\nu_G$ has $H = 0$.

\end{proof}

\bibliographystyle{amsplain}
\bibliography{Biblio}
 \end{document}